\documentclass[a4paper,12pt]{article}

\usepackage{cmap}

\usepackage{amsmath,amsthm,amssymb}
\usepackage[T2A]{fontenc}
\usepackage[english]{babel}
\usepackage[cp1251]{inputenc}
\usepackage{amscd}

\usepackage{graphics}

\usepackage{caption}
\usepackage{subcaption}

\usepackage{hyperref}

\usepackage[all,cmtip]{xy}

  \newskip\prethm \prethm3.0pt plus1.3pt minus.4pt
  \newskip\posthm \posthm2.7pt plus1.4pt minus.3pt
  \newtheoremstyle{STATEMENT}%
       {\prethm}{\posthm}{\itshape}{\parindent}{\scshape}
       {.}{.6em plus.2em minus.1em}{}
  \newtheoremstyle{EXPLANATION}%
       {\prethm}{\posthm}{}{\parindent}{\scshape}
       {.}{.6em plus.2em minus.1em}{}
\theoremstyle{STATEMENT}

\newtheorem {theorem}{Theorem}
\newtheorem {lemma}{Lemma}

\newtheorem    {prop} { Proposition}
\newtheorem    {claim} {Assertion}
\newtheorem    {corollary} {Corollary}

\theoremstyle{EXPLANATION}

\newtheorem {definition} {Definition}
\newtheorem   {remark}{Remark}
\newtheorem    {example} {Example}
\newtheorem    {problem} {Problem}

\begin{document}

\author{Kozlov~I.\,K.\thanks{No Affiliation, Moscow, Russia.} \\  ikozlov90@gmail.com}
\date{}

\title{Classification of Lagrangian fibrations}

\maketitle

\footnotetext[0]{This research was carried out with the financial support of the Russian Foundation for Basic
Research (grant no. 10-01-00748-a), the Programme for the Support of Leading Scientific Schools
of the President of the RF (grant no. НШ-3224.2010.1) and the Programme of the Development
of Research Potentials in Higher Education (grant no. РНП 2.1.1.3704).}

\begin{abstract}
This paper classifies Lagrangian fibrations over surfaces with compact total spaces up to fiberwise symplectomorphism identical on the base. \end{abstract}

\begin{flushleft}
\end{flushleft}

\tableofcontents

\section{Introduction}
\label{s1}

This paper classifies Lagrangian fibrations. A locally trivial fibration $\pi\colon(M^{2n},\omega)\to B$ is a
\textit{Lagrangian fibration} if the total space $(M^{2n},\omega)$ is a symplectic manifold and the symplectic form $\omega$ vanishes on each fibre $\omega\big|_{F^n}\equiv0$ (each fibre $F^n\subset M^{2n}$ is a Lagrangian submanifold). We assume that all fibres of Lagrangian fibrations are connected.

Lagrangian fibrations naturally arise in physics and mechanics. In particular, a momentum mapping of an integrable Hamiltonian system defines a Lagrangian fibration with singularities (about singularities of IHS, see for example \cite{1}). In this paper Lagrangian fibrations have no singularities.

In \cite{2}, using sheaf theory, J. Duistermaat introduced classifying invariants of Lagrangian fibrations : a lattice on the base and a Lagrangian Chern class (see \cite{2}, \cite{3}, and brief description in sections \ref{s2.4.2},~\ref{s2.4.3}). These invariants are hard to compute (even for low-dimensional fibrations). Nevertheless Mishachev \cite{4} obtained a full list of Lagrangian fibrations over orientable two-dimensional surfaces. Mishachev also show that among surfaces only the torus $\mathbb T^2$ and the Klein bottle $\mathbb K^2$ can be a base of a Lagrangian fibration.

In this paper we obtain a full list of Lagrangian fibrations over non-orientable surfaces. Any Lagrangian fibration (with compact fibre) over the Klein bottle can be obtained from a fibration $\pi\colon(I^2\times\mathbb T^2,\omega)\to I^2$ over a parallelogram $I^2$ by identifying points over opposite sides of $I^2$ by some fiberwise symplectomorphisms $g$ and $h$.

The base $I^2$ is a parallelogram with vertices $(0,0)$, $(x_0,0)$, $\bigl(\frac{m-\delta}2y_0,y_0\bigr)$ and $\bigl(x_0+\frac{m-\delta}2y_0,y_0\bigr)$ for some $x_0,y_0\in\mathbb R$, $m,\delta\in\mathbb Z$. The form $\omega$ is standard, that is $d\alpha\wedge dx+d\beta\wedge dy$, where $x$, $y$ are coordinates on the base and $\alpha,\beta\ (\operatorname{mod}1)$ are coordinates in the fibre. Fiberwise symplectomorphisms $g$, $h$ are defined by the formulas:
\begin{align}
\notag
h(x, y, \alpha, \beta) &= \biggl(x+my +\frac{m-\delta}{2}y_{0},
\ y + y_{0},
\ \alpha+ \frac{m_{0}}{x_{0}}x+\biggl(\frac{mm_{0}}{x_{0}} + \frac{n_{0}}{x_{0}}\biggr)y,
\\
\label{eq1}
&\qquad\qquad
- m \alpha+ \beta+ \frac{n_{0}}{x_{0}}x\biggr),
\\
\label{eq2} g(x, y, \alpha, \beta)
&= \biggl(x + (\delta-m) y +\frac{m-\delta}{2} y_{0} + x_{0},
\ -y+y_{0},
\ \alpha,
\ (\delta-m)\alpha-\beta\biggr)
\end{align}
for some $m_{0}, n_{0} \in \mathbb{Z}$.

Any Lagrangian fibration over the Klein bottle is fiberwise symplectomorphic to one of the fibrations
$$
\pi\colon(M^4_{m_0,n_0},\omega)\to\mathbb K^2_{m,y_0;\delta,x_0}
$$
described above (see Lemma \ref{l14}).

Note that some fibration in the list are fiberwise symplectomorphic.

Moreover, in this paper we classify all Lagrangian fibrations up to Lagrangian equivalence, that is up to fiberwise symplectomorphism identical on the base. Our strategy is to obtain (using obstruction theory) classifying invariants first and then to compute these invariants for fibrations over two-dimensional surfaces.

It was proved the following (see Theorems \ref{t1}--\ref{t5}). We start with classifying invariants.

An important invariant of a Lagrangian fibration $\pi\colon(M^{2n},\omega)\to B^n$ is a closed lattice $P$ (of rank $k$) in the cotangent bundle $T^*\!B^n$. A lattice $P$ is a multi-valued closed $1$-form such that its intersection with each cotangent space $T^*_xB^n$ is a discrete subgroup (of rank $k$). Since all fibres of $P$ are discrete, the quotient of the cotangent bundle $T^*\!B/P$ is well defined. Moreover, for any section $\alpha\colon B^n\to T^*\!B/P$ the differential $d\alpha$ is also well defined (the differential $d\alpha$ is the same for all local representatives of the form $\alpha$, since the lattice $P$ is closed).

The lattice of a Lagrangian fibration is defined as follows (this is a sketch of definition, for details see definition \ref{d9}). Each covector $\alpha\in T^*_xB^n$ defines an automorphism $s_\alpha\colon F_x\to F_x$ of the corresponding fibre $F_x$. If $\alpha=df\big|_x$, then $s_\alpha$ is the time-one flow of the vector field  $\operatorname{sgrad}(\pi^*f)$. The lattice $P$ consists of all covectors $\alpha$ such that $s_\alpha$ is the identity. It is easy to show that if the lattice has rank $k$, then fibres are diffeomorphic to $\mathbb T^k\times\mathbb R^{n-k}$.

The lattice $P$, considered as a locally trivial fibration over $B^n$, is canonically isomorphic to the local coefficients $\{\pi_1(F)\}$. Therefore instead of $H^2(B^n,\{\pi_1(F)\})$ we write $H^2(B,P)$ for the set of primary obstruction classes. If two Lagrangian fibrations $\pi_i\colon (M^{2n}_i,\omega_i)\to B^n$ over the same base $B^n$ have equal lattices $P_i$, then their primary obstruction classes belong to the same group $H^2(B^n,P)$ and consequently can be compared to each other (for details, see section \ref{s2.4.1}).

A lattice $P$ and a primary obstruction class $\gamma\in H^2(B,P)$ determine a Lagrangian fibration up to
the pullback of a $2$-form on the base. It was proved the following.

\begin{theorem} \label{t1}
Two Lagrangian fibrations $\pi_i\colon(M^{2n}_i,\eta_i)\to B^n$, $i=1,2$, with the same lattices $P_i\subset T^*\!B$ have equal primary obstruction classes $\gamma_i\in H^2(B,P_i)$ if and only if there exists a $2$-form $\varphi$ on the base $B$ such that the fibration $\pi_1\colon(M^{2n}_1,\eta_1+\pi_1^*\varphi)\to B^n$ is Lagrangian equivalent to the second fibration $\pi_2\colon(M^{2n}_2,\eta_2)\to B^n$.
\end{theorem}

The next statement is simple.

\begin{theorem}
\label{t2}
1) For any Lagrangian fibration $\pi\colon(M^{2n},\eta)\to B^n$ and any two-form $\varphi$ on the base the fibration $\pi\colon(M^{2n},\eta+\pi^*\varphi)\to B^n$ is Lagrangian if and only if the $2$-form $\varphi$ is closed.

2) Two Lagrangian fibrations $\pi\colon(M^{2n},\eta+\pi^*\varphi_i)\to B^n$, $i=1,2$, with lattice $P$ are Lagrangian equivalent if and only if $\varphi_1-\varphi_2=d\alpha$ for some section $\alpha\colon B^n\to T^*\!B^n/P$.
\end{theorem}

Not all lattices $P\,{\subset}\, T^*\!B^n$ and obstructions $\gamma\,{\in}\, H^2(B,P)$ can be realized by Lagrangian fibrations. It was proved the following.

\begin{theorem}
\label{t3} Every closed lattice $P\subset T^*\!B^n$ and every obstruction class $\gamma\in H^2(B,P)$ can be realized by a locally trivial fibration $\pi\colon(M^{2n},\eta)\to B^n$ such that the form $\eta$ vanishes on each fibre (that is $\eta|_{F_x}\equiv0$) and $d\eta=\pi^*\psi$ for some $3$-form $\psi$ on the base $B^n$.

The class $[\psi]\in H^3(B^n,\mathbb R)$ is uniquely defined.
\end{theorem}

Theorem \ref{t3} implies that a lattice and an obstruction class can be realized by a Lagrangian fibration  $\pi\colon(M^{2n},\eta)\to B^n$ if and only if the symplectic obstruction class $[\psi]\in H^3(B^n,\mathbb R)$ is trivial.

There do exist fibrations with nontrivial symplectic obstructions $[\psi]$ (see Example \ref{ex8}). Classification of  Lagrangian fibrations is naturally generalized to almost Lagrangian fibrations $\pi\colon(M^{2n},\eta)\to B^n$ (see definition \ref{d5}) whose fibres are not necessarily compact, and whose form $\eta$ is not necessarily closed. Obtained results are summarized in Theorem \ref{t11} that generalizes Theorems \ref{t1},~\ref{t2} and \ref{t3}.

If $H^2(B,\mathbb R)=H^3(B,\mathbb R)=0$, then Lagrangian fibrations over the base $B$ are classified by their lattices and primary obstruction classes.

In this paper we compute the invariants for all two-dimensional surfaces. Theorem \ref{t15} classifies all closed lattices of rank $2$ on compact two-dimensional surfaces up to isomorphism, that is up to an automorphism of the base. The next example contains the full list of lattices on the Klein bottle $\mathbb K^2$.

\begin{example}
\label{ex1}
Consider the affine plane $\mathbb R^2$. Let $G_{m,y;\delta,x}$ be a subgroup of $\mathrm{GL}_2(\mathbb
Z)\leftthreetimes\mathbb R^2\subset\mathrm{Aff}_2$ spanned by
$$\biggl(\begin{pmatrix}1&\delta
\\
0&-1\end{pmatrix},\begin{pmatrix}x
\\
0\end{pmatrix}\biggr), \quad
\biggl(\begin{pmatrix}1&m
\\
0&1\end{pmatrix},\begin{pmatrix}0
\\
y\end{pmatrix}\biggr),
\qquad
m,\delta\in\mathbb Z,
\quad
x,y\in\mathbb R,
\quad
x,y>0.
$$
The quotient $\mathbb R^2/G_{m,y;\delta,x}$ is diffeomorphic to the Klein bottle $\mathbb K^2$. Denote by $\mathbb K^2_{m,y;\delta,x}$ the lattice on the quotient $\mathbb R^2/G_{m,y;\delta,x}\simeq\mathbb K^2$ induced by the standard lattice $\mathbb Z^2$ on the plane.
\end{example}

\begin{theorem}
\label{t4}
For any closed lattice $P$ of rank $2$ on the Klein bottle $\mathbb K^2$ there exist unique $m,\delta\in\mathbb Z$, $x,y\in\mathbb R$, such that  $m\ge0$, $x,y>0$, $\delta$ is either $0$  or $1$, $m$ is even for $\delta=1$, and the lattice $P$ and the lattice $\mathbb K_{m,y;\delta,x}$ are isomorphic.
\end{theorem}

Other invariants are computed in section \ref{s4.3} for each closed lattice $P$ of rank 2 (see Theorem \ref{t16}). In particular, it is shown that for some lattices on the Klein bottle $\mathbb K^2$ there is only a finite number of different Lagrangian fibrations.

\begin{theorem}
\label{t5}
For any closed lattice $P$ isomorphic to the lattice $\mathbb K_{n,y;\delta,x}$, where $n\,{>}\,0$, there exist precisely $2n$ Lagrangian fibrations $\pi\colon(M^4,\omega)\to\mathbb K^2$ with lattice $P$ not Lagrangian equivalent to each other.
\end{theorem}

And for any lattice $P$ isomorphic to $\mathbb K_{0,y;\delta,x}$ there is a countable infinity of Lagrangian fibrations $\pi\colon(M^4,\omega)\to\mathbb K^2$ with lattice $P$.

\smallskip

Author thanks his scientific advisor A. \,A. Oshemkov for posing the problem and for help during the preparation of the paper.

\section{Basic definitions}
\label{s2}

\subsection{Almost Lagrangian fibrations}
\label{s2.1}

\begin{definition}
\label{d1}
Let $\eta$ be a non-degenerate $2$-form on $M^{2n}$. A submanifold $L\subset(M^{2n},\eta)$ is a \textit{maximal isotropic submanifold} if $\dim L=n$ and $\eta\big|_L=0$.
\end{definition}

Thus maximal isotropic submanifolds generalize maximal isotropic (and, in particular, Lagrangian) subspaces from linear algebra.

\begin{definition}
\label{d2}
A locally trivial fibration $\pi\colon(M^{2n},\eta)\to B^n$ is called a \textit{maximal isotropic fibration} if the fibres $F^{n} \subset (M^{2n}, \eta)$ are (connected) maximal isotropic submanifolds.
\end{definition}

Thus a maximal isotropic fibration is Lagrangian if and only if $d \eta =0$.

\begin{example}
\label{ex2}
Consider a cotangent bundle $\pi\colon T^*\!B\to B$. Let $\xi$ be the \textit{canonical (tautological) $1$-form} on the cotangent bundle $T^*\!B$ defined by $\langle \xi, v\rangle = \alpha(\pi_*(v))$ for any vector $v\in T_{(x,\alpha)}(T^*\!B)$ at a point $(x, \alpha) \in T^*_xB$. Then the cotangent bundle $\pi\colon(T^*\!B,\omega_0)\to B$ with the \textit{canonical $2$-form} $\omega_0=d\xi$ is a maximal isotropic (even Lagrangian) fibration.
\end{example}

It is easy to show that in \textit{canonical coordinates} $(p_i,q^i)$, where $q^i$ are coordinates on the base and $p_i$ are coordinates of a covector with respect to the basis $dq^i$, the symplectic form is standard $\omega_{0}= \sum_idp_i\wedge dq^i$.

\begin{example}
\label{ex3}
Let $\pi\colon(M^{2n},\eta)\to B^n$ be a maximal isotropic fibration and $\varphi$ be a $2$-form on $B^n$. We say that a maximal isotropic fibration $\pi\colon(M^{2n},\allowbreak \eta+\nobreak\pi^*\varphi)\to B^n$ is a fibration \textit{twisted by the form} $\varphi$.
\end{example}

\begin{definition}
\label{d3}
Given two fibrations $\pi_i\colon(M^{2n}_i,\eta_i)\to B^n_i$, $i=1,2$. A fiberwise diffeomorphism $f\colon M^{2n}_1\to M^{2n}_2$ is a \textit{Lagrangian isomorphism} if $f^*\eta_2=\eta_1$. If $B_{1}=B_{2}$ and $f$ is identical on the base then $f$ is a \textit{Lagrangian equivalence}.
\end{definition}

This paper classifies Lagrangian fibrations up to Lagrangian equivalence. The classification is based on the following Duistermaat construction (see \cite{2}). For any covector $\alpha\in T^*_xB$ on the base of a Lagrangian fibration $\pi:(M^{2n}, \eta) \rightarrow B^{n}$ Duistermaat constructed a vector field $v_\alpha$ on the fibre $F_x=\pi^{-1}(x)$. Here is a generalization of this construction to the case of maximal isotropic fibrations.

\subsubsection*{Duistermaat construction}

Consider a surjective map $\pi\colon E\to B$. Let $\nu$ be a bivector field on $E$, that is $\nu$ is a skew-symmetric bilinear form on each cotangent space $T_z^*E$ that varies smoothly with $z\in E$. If $E$ is a Poisson manifold, then $\nu$ is the Poisson bivector.

For any covector $\alpha\in T^*_xB$ and any point $z\in F_x=\pi^{-1}(x)$ there corresponds a vector $v_\alpha(z)\in T_zE$ symplectically dual (with respect to $\nu$) to the covector $(\pi^*\alpha)(z)$. In other words, the vector $v_\alpha(z) \in T_zE$ is defined by $\langle v_\alpha(z),\beta\rangle=\nu((\pi^*\alpha)(z),\beta)$ for all $\beta\in T^*_zE$.

Suppose that $F_x$ is a submanifold of the total space $E$, and that the bivector $\nu$ vanishes on the subspace $\pi^*(T^*_xB)\subset T^*_zE$. Then each vector $v_\alpha(z)$ is tangent to $F_x$. Hence $v_\alpha$ is a vector field on $F_x$.

\begin{definition}
\label{d4} The vector fields $v_\alpha$ are \textit{shiftings}. We denote by $s_\alpha^t$ the flow of the vector field $v_\alpha$ for time $t$ (if the flow is well defined). The time-one flow $s_\alpha^1$ is denoted by $s_\alpha$ for short.
\end{definition}

\begin{remark}
\label{r1}
Now, let $\pi\colon E\to B$ be a fibration with a non-degenerate $2$-form $\eta$ on $E$. Then Duistermaat construction can be applied to the bivector field $\nu$ dual to the form $\eta$. (In textbooks on linear algebra it is shown that any non-degenerate bilinear form $\Phi$ on a vector space $V$ defines the dual form $\widetilde\Phi$ on the dual space $V^*$).

If $\pi\colon E\to B$ is a maximal isotropic fibration, then $\nu\big|_{\pi^*(T^*\!B)}=0$ (check it in local coordinates). Hence the correspondence $\alpha\to v_\alpha$ defines an isomorphism between the cotangent space $T^*_{\pi(z)}B$ at $\pi(z)$ and the tangent space $T_zF$ to the fibre $F=\pi^{-1}(\pi(z))$ at $z\in E$.

In particular, the shifting of a covector $\alpha=df\big|_x$, where $f$ is a function on $B$, is given by $v_\alpha=-\operatorname{sgrad}\pi^*f$. Recall that, the skew-symmetric gradient $\operatorname{sgrad} g$ of a function $g$ is defined by the formula $i_{\operatorname{sgrad} g}\eta=-dg$, where $i_{v_1}\eta(v_2)=\eta(v_1,v_2)$ for all vectors $v_1$ and $v_2$.
\end{remark}

In this paper we classify Lagrangian fibrations with complete shiftings $v_\alpha$. Moreover, we classify all \textit{almost Lagrangian fibrations}, that is fibrations $\pi\colon (E, \eta)\to B$ such that the form $\eta$ can be pushed down on the base ($d\eta=\pi^*\psi$ for some $3$-form $\psi$ on the base) and all shiftings commute (see definition \ref{d4}). But first let us show that almost Lagrangian fibrations are precisely fibrations which possess a fiberwise action of the cotangent bundle (defined by $\alpha\to s_\alpha$), and which are locally Lagrangian equivalent to a twisted cotangent bundle.

The correspondence $\alpha\to s_\alpha$ defines a fiberwise action if and only if the following holds: first, all shiftings are complete; secondly, all shiftings commute. If the shiftings are complete, then $s_\alpha$ is well defined, and if the shiftings commute, then the map $\alpha \to s_\alpha$ is a homomorphism of groups.

Let us prove that if a fibration is locally Lagrangian equivalent to a twisted cotangent bundle, then all shiftings commute. (It follows from Assertions \ref{cl1} and \ref{cl2}.)

\begin{claim}
\label{cl1} Let $\pi\colon(M^{2n},\eta)\to B^n$ be a maximal isotropic fibration and $z$ be a point of the total space $M^{2n}$. Then the following are equivalent.
\begin{itemize}
\item[{\rm 1)}]
In a neighbourhood of $z$ all shiftings $v_\alpha$ commute.
\item[{\rm 2)}]
There exist local coordinates $p_1,\dots,p_n,q^1,\dots,q^n$ at $z$, where $q^i$ are coordinates on the base and $p_i$ are coordinates in the fibre, such that the matrix of the form $\eta$ is $\bigl(\begin{smallmatrix}0&&E
\\ -E&&\Omega\end{smallmatrix}\bigr)$. Here $E$ denotes the $(n\times n)$ identity matrix.
\end{itemize}\end{claim}

\begin{proof}

1)~$\Rightarrow$~2). Choose local coordinates $q^1,\dots,q^n$ on $B$ (and consider them as coordinates on a neighbourhood of $z$). We want to extend $q^i$ to local coordinates at $z$, that is to find coordinates $p_1,\dots,p_n$ such that $v_{dq^i}=\frac\partial{\partial p_i}$. Consider a (local) section $s\colon B\to E$ and let $p_i$ be the time of the shift of the section $s$ along vector fields $v_{dq^i}$. Since (linearly independent) vector fields $v_{dq^i}$ commute, coordinates $p_i$ are well defined.

2)~$\Rightarrow$~1). In given coordinates $p_1,\dots,p_n,q^1,\dots,q^n$ the shifting of a covector $\alpha=\alpha_i\,dq^i$ has the form $\alpha_i\frac\partial{\partial p_i}$. Vector fields with constant coefficients commute.
\end{proof}

\begin{claim}
\label{cl2}
Let $\pi\colon(M^{2n},\eta)\to B^n$ be a maximal isotropic fibration and $z$ be a point of the total space $M^{2n}$. Then the following are equivalent.
\begin{itemize}
\item[{\rm 1)}] A neighbourhood of $z$ is Lagrangian equivalent to a neighbourhood of $\pi(z)$ in a twisted cotangent bundle.
\item[{\rm 2)}] There exist a neighbourhood of $z$ such that $d\eta=\pi^*\psi$ for some $3$-form $\psi$ on the base.
\item[{\rm 3)}] There exist local coordinates $p_1,\dots,p_n,q^1,\dots,q^n$ at $z$ where $q^i$ are coordinates on the base and $p_i$ are coordinates in the fibre, such that the matrix of the form $\eta$ is $\bigl(\begin{smallmatrix}0&&E \\ -E&&\Omega\end{smallmatrix}\bigr)$. Here $E$ denotes the $(n\times n)$ identity matrix and the matrix $\Omega$ depends only on $q^1,\dots,q^n$.
\end{itemize}
\end{claim}

\begin{proof}
3)~$\Rightarrow$~1). As noted after example \ref{ex2},
$\bigl(\begin{smallmatrix}0&E
\\
-E&0\end{smallmatrix}\bigr)$ is the matrix of $\omega_0$ in canonical coordinates on a cotangent bundle. $\Omega$ is the matrix of the twisting form.

1)~$\Rightarrow$~2). The form $\psi$ is the differential of the twisting form.

2)~$\Rightarrow$~3). Let us show that there exist coordinates $p_1,\dots,p_n,\allowbreak q^1, \dots,q^n$ such that the matrix of the form $\eta$ has the form $\bigl(\begin{smallmatrix}0&&E
\\
-E&&\Omega\end{smallmatrix}\bigr)$.  It suffices to show that all shiftings commute (see Assertion \ref{cl1}).
The proof is divided into two steps.

\textsl{Step}~1. Suppose that $d\eta=0$. Then the corresponding Poisson bracket is well defined. From the Jacobi identity it follows that
$$
[v_{dq^i},v_{dq^j}]=[\operatorname{sgrad}\pi^*q^i,
\operatorname{sgrad}\pi^*q^j]=-\operatorname{sgrad}\{\pi^*q^i,\pi^*q^j\}=\operatorname{sgrad}0.
$$
The last equality can be easily checked in local coordinates. If the matrix of the $2$-form $\eta$ has the form  $\bigl(\begin{smallmatrix}0&A
\\
-A^T&B\end{smallmatrix}\bigr)$, then the matrix of the corresponding Poisson bracket (the inverse matrix) is  $\bigl(\begin{smallmatrix}(A^T)^{-1}BA^{-1}&-(A^T)^{-1}
\\
A^{-1}&0\end{smallmatrix}\bigr)$.

\textsl{Step} 2. General case $d\eta=\pi^*\psi$. Since $d^2\eta=0$, we have $d\psi=0$. Therefore $\psi=d\varphi$ for some $2$-form $\varphi$ on the base $B$. But a twisting does not change shiftings (shiftings depend only on the matrix $A$ from the previous step and twistings do not change this matrix). Thus shiftings commute.

It remains to prove that the matrix $\Omega$ does not depend on $p_i$. This is true, since $d\eta=\pi^*\psi$.
\end{proof}

Let us show that if a fibrations is locally equivalent to a twisted cotangent bundle and all shiftings are complete, then there exists a (global) $3$-form $\psi$ on the base $B^n$ such that $d\eta=\pi^*\psi$.

\begin{claim}
\label{cl3}
Let $\pi\colon (M^{2n},\eta)\to B^n$ be a maximal isotropic fibration which satisfies the (equivalent) conditions of Assertion \ref{cl1}. If all shiftings $v_\alpha$ are complete, then the following are equivalent.
\begin{itemize}
\item[\rm 1)] Suppose that $U$ is a simply connected subset of $B^n$. Then for any two sections $u_1,u_2\colon U\to M^{2n}$ there exists a $1$-form $\alpha$ on $U\subset B^n$ such that the time-one flow of the $v_\alpha$ takes $u_1$ to $u_2$. In other words, $s_\alpha\circ u_1=u_2$.
\item[\rm 2)]
For any $1$-form $\alpha$ on a neighbourhood of $x\in B^n$ we have $s_\alpha^*\eta=\eta+d(\pi^*\alpha)$.
\item[\rm 3)]
There exists a $3$-form $\psi$ on the base $B^n$ such that $d\eta=\pi^*\psi$.
\end{itemize}
\end{claim}

\begin{proof}
1) The correspondence $\alpha\,{\to}\, s_\alpha$ defines an action of $T^*_xB^n$ on $F_x$.
The action is transitive, since vector fields $v_\alpha$ span each cotangent space $T_zF$ of the fibre $F$
at point $z\in E$. Hence for any point $x\in U$ there exists a covector $\beta\in T^*_xB^n$ such that $s_\beta(u_1(x))=u_2(x)$. Extend the covector $\beta$ to a $1$-form $\alpha$ on $U$ along paths
(it is possible, since $U$ is simply connected).

2) The Lie derivative of the form $\eta$ along vector field $v_\alpha$ is given by Cartan's formula
\begin{equation}
\label{eq3} L_{v_{\alpha}} \eta= i_{v_{\alpha}}(d \eta) +
d(i_{v_{\alpha}}\eta).
\end{equation}
The first term is zero (this is clear in local coordinates from condition 3) of Assertion \ref{cl2}). By definition of $v_\alpha$, the last term is equal to $d(\pi^*\alpha)$.

3) Choose any (local) section of $M^{2n}$ and let $\psi$ be the pull back of the form $d\eta$. By 1) and 2) the form $\psi$ is well defined.
\end{proof}

\begin{definition}
\label{d5}
A maximal isotropic fibration $\pi\colon(M^{2n},\eta)\to B^n$ is an \textit{almost Lagrangian fibration} if all shiftings $v_\alpha$ are complete and $d\eta=\pi^*\psi$ for some $3$-form $\psi$ on the base $B^n$.
\end{definition}

We assume that all shiftings of Lagrangian fibrations are complete (this means that all Lagrangian fibrations are almost Lagrangian).

\begin{remark}
\label{r2}
All shiftings are complete if the fibres are compact. The fibres are compact if the total space $M^{2n}$ is compact.
\end{remark}

\begin{definition}
\label{d6}
If $d\eta=\pi^*\psi$, then the class $[\psi]\in H^3(B,\mathbb R)$ is a \textit{symplectic obstruction class}.
\end{definition}

The next lemma clarifies the definition.

\begin{lemma}
\label{l1}
An almost Lagrangian fibration has a trivial symplectic obstruction if and only if it is a twisted Lagrangian fibration.
\end{lemma}

\begin{proof}
Let $\psi=d\varphi$. Twist the fibration by $-\varphi$ to get a Lagrangian fibration.
\end{proof}

\subsection{Affine fibrations}
\label{s2.2}

The classification of almost Lagrangian fibrations is based on the observation that for any covector $\alpha$ on the base there corresponds an automorphisms of the fibre (time-one flow of the shifting $v_\alpha$). In this section we formalize the notion of a fiberwise action of one fibrations on another fibration.

\begin{definition}
\label{d7}

Consider a locally trivial fibration $\pi_0\colon V\to B$ in which each fibre is an abelian group (the structure of a group depends smoothly on the fibre).

We say that the fibration $\pi_0\colon V\to B$ \textit{acts fiberwisely} on a fibration $\pi\colon E\to B$ if each fibre $W_x=\pi_0^{-1}(x)$ acts on the corresponding fibre  $F_x=\pi^{-1}(x)$.

We say that $\pi\colon E\to B$ is an \textit{affine fibration} with the \textit{underlying fibration} $\pi_0\colon V\to B$ if each action of $W_x$ on $F_x$ is regular (that is free and transitive).

Two affine fibrations with the fiberwise action of  $\pi_0\colon V\to B$ are \textit{equivalent} if there exists a fiberwise diffeomorphism of total spaces which preserves the action of $\pi_0\colon V\to B$.

\end{definition}

From now on we use the notation $v\circ x$ for the action of $v\in W_x$ on $x\in F_x$.

\begin{remark}
\label{r3}
In the sequel, the fibration $\pi_0\colon V\to B$ is $T^*\!B$ but we do not use the scalar multiplication.
\end{remark}

Suppose that an affine fibration $\pi\colon E\to B$ with the underlying fibration $\pi_0\colon V\to B$ has a section $s\colon B\to E$. By definition of an affine fibration, for any two points $z_1$, $z_2$ of a fibre $F_x\subset E$ there exists one and only one element $v\in W_x$ which takes $z_1$ to $z_2$. It follows that the correspondence which assigns to each element $\alpha\in W_x$ the point $\alpha\circ s(x)$ defines a canonical equivalence of the affine fibration $\pi\colon E\to B$ and the underlying fibration $\pi_0\colon V\to B$. Note that without the section $s\colon B\to E$ there is no such canonical equivalence.

This suggests the following definition.

\begin{definition}
\label{d8}
A fibration $\pi\colon E\to B$ is an \textit{affine fibration} with the \textit{underlying fibration} $\pi_0\colon V\to B$ if there exists a cover $\{U_i\}$ of the base $B$ and maps $\phi_i\colon \pi^{-1}(U_i)\to\pi_0^{-1}(U_i)$ such that each transition function $\phi_i\circ\phi_j^{-1}$ is an action of a section $V$ over $U_i\cap U_j$.
\end{definition}

Definitions \ref{d7} and \ref{d8} are equivalent. Definition \ref{d8} was derived from definition \ref{d7}. On the other hand, any affine fibration from definition \ref{d8} has a natural fiberwise action of the underlying fibration over $\pi_0^{-1}(U_i)$. By definition, this action does not depend on the neighbourhood $U_i$.

Each almost Lagrangian fibration admits a fiberwise action of the cotangent bundle $T^*\!B$. Next definition describes the isotropy group of the action (the set of all covectors which act trivially on the fibre).

\begin{definition}
\label{d9} A \textit{lattice of rank $k$} on a manifold $B$ (or in $T^*\!B$) is a submanifold of $T^*\!B$ transversal to the fibres of $T^*\!B$ such that its intersections with each fibre $T^*_xB$ is a discrete subgroup of rank $k$.

Two lattices on $B$ and $B'$ are \textit{isomorphic} if there exists a diffeomorphism $f\colon B\to B'$ which differential $df$ moves one lattice to the other.

A lattice $P$ is \textit{closed} if any (local) section of $P$ is a closed $1$-form.
\end{definition}

In other words, a lattice on $B$ is a lattice $\mathbb Z^k$ in every cotangent bundle $T^*_xB$ smoothly varying with $x$.

By $(B,P)$ denote a manifold $B$ with a lattice $P$ on it.

In the sequel we often use the following assertion without reference.

\begin{claim}
\label{cl4} Let $\pi\colon E\to B$ be an affine fibration with the underlying fibration $\pi_0\colon V\to B$. Then for any continuous map $f\colon \widetilde B\to B$ the induced fibration $f^*\pi=\widetilde\pi\colon \widetilde E\to\widetilde B$ is an affine fibration over $\widetilde B$ with the underlying fibration $f^*\pi_0=\widetilde \pi_0\colon \widetilde V\to\widetilde B$.
\end{claim}

The fibration  $f^*\pi=\widetilde\pi\colon\widetilde E\to\widetilde B$ is an \textit{affine fibration induced by $f$}.

\subsection{Equivalence of affine and almost Lagrangian fibrations}
\label{s2.3}

As we saw in section \ref{s2.1}, any almost Lagrangian fibration $\pi\colon(M^{2n},\eta)\to B^n$ admits a fiberwise action of the cotangent bundle $T^*\!B^n$. The action of a covector $\alpha$ is the time-one flow along the corresponding shifting $v_\alpha$. Let us show that the action of the cotangent bundle $T^*\!B^n$ determines an almost Lagrangian fibration up to a twisting (twistings are defined in example \ref{ex3}).

\begin{theorem}
\label{t6}
1) Any almost Lagrangian fibration $\pi\colon(M^{2n},\eta)\to B^n$ is an affine fibration with the underlying fibration $T^*\!B/P$ for some closed lattice $P$ on $B^n$. The fiberwise action of the cotangent bundle is defined by the correspondence $\alpha\to s_\alpha$.

2) Any affine fibration $\pi\colon E\to B$ with the underlying fibration $T^*\!B/P$ can be realized by an almost Lagrangian fibration for any closed lattice $P$ on $B$.

3) Two almost Lagrangian fibrations are Lagrangian equivalent up to a twisting if and only if the corresponding affine fibrations are equivalent.
\end{theorem}

In other words, an almost Lagrangian fibration considered up to a twisting and an affine fibration with the underlying fibration $T^*\!B/P$ for some closed lattice $P$ on $B$ are the same thing.

\begin{proof}[of Theorem~\ref{t6}]
1) The correspondence $\alpha\to s_\alpha$ defines a fiberwise action of $T^*\!B$ on the given almost Lagrangian fibration. By assertion \ref{cl3}, the action is transitive (in each fibre). Since the dimension the base is equal to the dimension of the fibre, the isotropy group $P$ is a discrete subgroup in each cotangent space $T^*_xB$. The rank of $P$ is constant, since the fibration is locally trivial. Thus $P$ is a lattice. It follows from Assertion \ref{cl3} that $P$ is closed.

2) Let $\pi\colon E\to B$ be an affine fibration with the underlying fibration $\pi_0\colon T^*\!B/P\to B$. We must prove that the fibration $\pi\colon E\to B$ admits a structure of an almost Lagrangian fibration which is compatible with the given structure of an affine fibration. In other words, we must find a non-degenerate $2$-form $\eta$ vanishing on each fibre such that $d\eta=\pi^*\psi$ (for some $3$-form $\psi$ on the base $B$) and $s_\alpha z=\widetilde\alpha\circ z$ for all points $z\in E$ and all covectors $\alpha\in T^*_{\pi(z)}B$ (here $\widetilde\alpha$ is the element of $T^*\!B/P$ corresponding to $\alpha\in T^*\!B$). Consider the cover $\{U_i\}$ of $B$ from definition \ref{d8}. Denote by $\eta_i$ the $2$-form on $\pi^{-1}(U_i)\simeq T^*U_i/P$ induced by the canonical $2$-form $\omega_0$ on $T^*U_i$. The form $\eta$ is well defined, since the lattice is closed. Now choose a partition of unity $\phi_i$ which is subordinate to the cover $\{U_i\}$. The form $\eta=\sum_i\phi_i\eta_i$ is as required.

 The last statement can be easily checked in local coordinates. In coordinates $p_i, q_i$ at $z\in E$, where $q^1,\dots,q^n$ are coordinates on the base and $p_1,\dots,p_n$ are coordinates in the fibre with respect to the basis $dq^i$, the matrix of the form $\phi_i\eta_i$ has the form $\bigl(\begin{smallmatrix}
0&\phi_iE
\\
-\phi_iE&\phi_i\Phi_i
\end{smallmatrix}\bigr)$, where the $(n\times\nobreak n)$ matrix $\Phi_i$ depends only on $q^1,\dots,q^n$. Thus the matrix of $\eta$ is $\bigr(\begin{smallmatrix}0&E
\\
-E&\sum_i\phi_i\Phi_i\end{smallmatrix}\bigl)$.

3) It is clear that a twisting does not change the structure of an affine fibration. If $\pi_i\colon(M^{2n},\eta_i)\to B^n$, $i=1,2$, are two almost Lagrangian fibrations with the same underlying fibration, then there exists a $2$-form $\varphi$ on the base $B^n$ such that $\eta_1-\eta_2=\pi^*\varphi$. Let $\varphi$ be the pullback of the difference $\eta_1-\eta_2$ by (local) sections. It follows from Assertion \ref{cl3} that the form $\varphi$ is well defined.
\end{proof}

Now we need some definitions.

\begin{definition}
\label{d10}
A section of $T^*\!B/P$ is a \textit{latticed 1-form} (on $B$).
\end{definition}

By $\Omega_P^1(B)$ denote the space of latticed $1$-forms on $B$.

For any closed lattice $P$ the differential of a latticed $1$-form $\alpha$ is a well defined $2$-form on $B$. The form $d\alpha$ is the differential of any local representative of $\alpha$.

\begin{definition}
\label{d11}
A latticed $1$-form $\alpha$ is \textit{closed} if $d\alpha=0$.
\end{definition}

\begin{lemma}
\label{l2} Let $\pi\colon(M^{2n},\eta)\,{\to}\, B^n$ be an almost Lagrangian fibration and $\varphi_1$, $\varphi_2$ be $2$-forms on $B^n$. The twisted fibrations $\pi\colon(M^{2n},\eta+\pi^*\varphi_i)\to B^n$ are Lagrangian equivalent if and only if $\varphi_2-\varphi_1=d\alpha$ for some latticed $1$-form $\alpha$.
\end{lemma}

\begin{proof}
($\Leftarrow$) Let $\varphi_2-\varphi_1=d\alpha$. Choose a (local) representative $\widetilde\alpha$ of the latticed $1$-form $\alpha$. Consider the shifting $v_{\widetilde\alpha}$ of the $1$-form $\widetilde\alpha$ with respect to the form $\eta+\pi^*\varphi_1$. The time-one flow of $v_{\widetilde\alpha}$ takes the fibration twisted by $\varphi_1$ to the fibration twisted by $\varphi_2$. (The shift along the vector field $v_{\widetilde\alpha}$ does not depend on $\widetilde\alpha$, since $P$ is the isotropy group of the fiberwise action).

($\Rightarrow$)  Lagrangian equivalent fibrations determine equivalent affine fibrations, since a twisting does not change the structure of an affine fibration. An equivalence of affine fibrations with the underlying fibration $T^*\!B/P$ is a shift along a latticed $1$-form. It follows from \ref{cl3} that $\varphi_1-\varphi_2=d\alpha$.
\end{proof}

\subsection{Classification of affine fibrations}
\label{s2.4}

In this section we discuss several approaches to the classification of affine fibrations $\pi_i\colon E_i\to B$ with the underlying fibration $T^*\!B/P$ up to equivalence. The first approach uses obstruction theory and the second classification is in terms of transition functions. At the end of the section we describe the classification of Lagrangian fibration in terms of \v{C}ech cohomology (used by K. N. Mishachev in \cite{4} and by J. Duistermaat in \cite{2}).

\subsubsection{Classification in terms of obstructions}
\label{s2.4.1}

Fix the following structure of a cell complex on $B$. In this section the $i$-th $k$-cell of $B$ is denoted by $e^k_i$ and the corresponding characteristic map is denoted by $\chi_i^k\colon D^k\to B$. For any fibration $\pi\colon E\to B$ we denote by $\pi^l\colon E^l\to B^l$ the induced fibration over the $l$-skeleton $B^l$.

Let us prove that affine fibrations with the underlying fibration $T^*\!B/P$ are classified by their primary obstruction classes. We prove first that affine fibrations with equal obstructions are equivalent (see Theorem \ref{t7}). Then we prove that any obstruction can be realized by an affine fibration (see Theorem \ref{t8}).

Before proving this theorems, we recall briefly the definition of primary obstruction classes (for obstruction theory see, for example, \cite{5},~\cite{6}). Then we show how to compare primary obstruction classes for affine fibrations with the underlying fibration $T^*\!B/P$.

\subsubsection*{Obstructions}
Let $s\colon B^k\to Y$ be a map from the $k$-skeleton of $B^k$ to a homotopy $k$-simple topological space $Y$ (here \textit{homotopy $k$-simplicity} means that homotopy groups $\pi_k(Y,y_0)$ are canonically isomorphic for all points $y_0$ of $Y$). Then for any $(k+1)$-cell $e^{k+1}$ the restriction $s\big|_{\partial e^{k+1}}\colon S^k\to Y$ of the map $s$ determines an element $A\in\pi_k(Y)$. Using this idea one can define obstructions to an extension of the map.

But we need obstructions to the existence of a section, and hence instead of maps we take cross sections. Let $\pi\colon E\to B$ be a locally trivial fibration with homotopy $k$-simple fibre $F$, and $s\colon B^k\to E$ be a section over $B^k$. A locally trivial fibration over a disc is trivial, thus the induces fibration  $(\chi^{k+1}_i)^*E\to D^{k+1}$ over each cell $e^{k+1}_i$ is trivial. Consequently, the restriction of the section $s$ determines the map $s\big|_{\partial D^{k+1}}\colon S^k\to F$. We see that it is possible to assign to each cell $e^{k+1}_i$ an element $c_i\in\pi_k(F)$ of the homotopy group of the fibre $F$ over the cell.

Recall that a system of local coefficients on $B$ assigns to each $x \in B$ a group $G_{x}$, and assigns to each path $s$ from $x$ to $y$ a homomorphism $\gamma_s\colon G_x\to G_y$ of these groups. The homomorphism must be the same for homotopic paths and to a composition of paths there must correspond the composition of homomorphisms.

If the fibre is homotopy $k$-simple, then the set of $k$-th homotopy groups of fibres form a system of local coefficients, denoted by $\{\pi_k(F)\}$. Since all cells are simply connected (even contractible), all groups $\pi_k(F)$ over each cell are canonically isomorphic. The set of $c_i$ forms a cellular $(k+1)$-cochain $c\in\mathcal C^{k+1}(B,\{\pi_k(F)\})$ called \textit{the obstruction cochain}.

Each obstruction cochain $c$ is a cocycle. The class $[c]\in H^{k+1}(B,\{\pi_k(F)\})$ is called the \textit{obstruction to an extension of the section $s$ over the $(k+1)$-skeleton (or primary obstruction class)}. Any coboundary corresponds to a change of the section over the $k$-skeleton outside of the $(k-1)$-skeleton. More precisely, one has the following.

\begin{claim}
\label{cl5} Consider a locally trivial fibration $\pi\colon E\to B$. Let $s\colon B^k\to E^k$ be a section with obstruction cochain $c\in\mathcal C^{k+1}(B,\{\pi_k(F)\})$. Then for any cochain $\widetilde c$ such that $\widetilde c-c=\delta b$ for some $b\in\mathcal C^k(B,\{\pi_k(F)\})$ there exists a section $\widetilde s\colon
B^k\to E^k$ with obstruction cochain $\widetilde c$ such that the sections $s$ and $\widetilde s$ coincide on the $(k-1)$-skeleton $B^{k-1}$.
\end{claim}

\subsubsection*{Comparing obstructions}

Let $\pi\colon E\to B$ be an affine fibration with the underlying fibration $T^*\!B/P$. Then the lattice $P$ is canonically isomorphic to the system of local coefficients $\{\pi_1(F)\}$. Indeed, for any fibre $F_x$ the fundamental group $\pi_1(F_x)$ is canonically isomorphic to the corresponding fibre $P_x$ of the lattice $P$, since $\pi_1(F_x)\simeq\pi_1(T^*_xB/P_x)\simeq P_x$. It follows that if $\pi_i\colon E_i\to B$, $i=1,2$, are two affine fibrations with the underlying fibration $T^*\!B/P$, then their obstructions to an extension of a section to the $2$-skeleton lie in the same group $H^{2}(B, P)$. This means that we can compare these obstruction.

Analogously, we can compare obstruction cochains.

\begin{remark}
\label{r4}
All affine fibrations considered in this paper have fibres $F$ diffeomorphic to $\mathbb R^m\times\mathbb T^{n-m}$. It implies that the fibres are homotopy $k$-simple for all $k$.
\end{remark}

\begin{theorem}
\label{t7}
Two affine fibrations $\pi_i\colon E_i\to B$, $i=1,2$, with the underlying fibration $\pi_0\colon T^*\!B/P\to B$, where $P$ is a closed lattice on $B$, are equivalent if and only if they have equal primary obstruction classes $\gamma_1=\gamma_2\in H^2(B,P)$ (the same obstruction to the existence of a section over the $2$-skeleton).
\end{theorem}

\begin{proof}
($\Rightarrow$) Equivalent fibrations have equal obstructions.

($\Leftarrow$) Suppose that fibrations $\pi_i\colon E_i\to B$ have equal obstructions. Let us construct an equivalence skeleton by skeleton, cell by cell.

Suppose that we have already constructed an equivalence
$$
E^{k-1}_1\sim_{k-1}E^{k-1}_2.
$$
over the $(k-1)$-skeleton $B^{k-1}$.

We want to construct an equivalence over a $k$-cell $e^k_j$. An affine fibration with a section is canonically isomorphic to the underlying fibration (see section \ref{s2.2}). Therefore it is sufficient to find two sections  $s_i\colon D^k\to(\chi^k_j)^*E_i$, $i=1,2$, such that the equivalence which takes $s_1$ to $s_2$ coincides with the given equivalence $\sim_{k-1}$ over the $(k-1)$-skeleton.

Furthermore, in order to construct an equivalence over a cell $e^k_j$ it suffices to find a contractible section  $s^k_1\colon\partial D^k\to(\chi^k_j)^*E_1$ such that the corresponding section $s^k_2\colon\partial D^k\to(\chi^k_j)^*E_2$ with respect to $\sim_{k-1}$ is contractible (we say that a section $s$ over the boundary of a cell $\partial e^k_j$ is \textit{contractible} if this section can be extended to the whole cell $e^k_j$).

Let us consider three cases.

\textsl{Case $k=1$}. Both affine fibrations $\pi_i\colon E_i\to B$ admit a section  $s^1_i\colon B^1\to E^1_i$ over the $1$-skeleton. Moreover, by Assertion \ref{cl5}, we may assume that these sections $s^1_i$ determine equal obstruction cochains.

\textsl{Case $k=2$}. Suppose that the equivalence $\sim_1$ over the $1$-skeleton is defined by sections $s^1_i\colon B^1\to E_i$, $i=1,2$, over the $1$-skeleton $B^1$ with equal obstruction cochains. Let us prove that the equivalence $\sim_1$ can be extended to the $2$-skeleton.

Consider a $2$-cell $e^2_j$. Choose an arbitrary section $s_1$: $D^2\to(\chi^2_j)^*E_1$ over the cell $e^2_j$. We claim that the section $s_2\colon S^1\to(\chi^2_j)^*E_2$ corresponding to the section $s_1\big|_{\partial e^2_j}\colon
S^1\to(\chi^1_j)^*E_1$ with respect to $\sim_1$ is contractible. First, identify affine fibrations $(\chi^2_j)^*E_i\to D^2$ with the underlying fibrations $(\chi^2_j)^*(T^*\!B/P)\to D^2$. We may assume that the section $s_1\colon D^2\to E^1_1$ corresponds to the zero section.

Given two sections $s^1_i\colon B^1\to E_i$, denote by $\widetilde{s_i}\colon S^1\to(\chi^2_j)^*(T^*\!B/P)$, $i=1,2$, the corresponding sections of $(\chi^2_j)^*(T^*\!B/P)\to D^2$ over $\partial D^2$. These sections $\widetilde{s_1}$ and $\widetilde{s_2}$ define the same element of $\pi_1((\chi^2_j)^*(T^*\!B/P))$ (these sections determine equal obstructions). Hence the difference $\widetilde{s_2}-\widetilde{s_1}$ determines the identity element of $\pi_1((\chi^2_j)^*(T^*\!B/P))$. This means that the section $\widetilde{s_2}-\widetilde{s_1}$ is contractible. Thus $s_2$ is contractible, as required.

\textsl{Case $k>2$}. Let $(\chi^k_j|_{e^k_j})^*E_i\to D^k$ be a locally trivial fibration. Then obstructions to an extension of a section lie in the group $\pi_{k-1}(F)=\pi_{k-1}(\mathbb T^m\times\mathbb
R^{n-m})$ which is trivial for $k>2$. Thus any section over $\partial e^k_j\simeq S^{k-1}$ can be extended to a section over the whole cell $e^k_j\simeq D^k$.
\end{proof}

Any obstruction can be realized by an affine fibration.

\begin{theorem}
\label{t8} Let $P$ be a closed lattice on $B$. For any class $\gamma\in H^2(B,P)$ there exists an affine fibration with the underlying fibration $T^*\!B/P$ and with primary obstruction class $\gamma$.
\end{theorem}

\begin{proof}

Choose a representative cocycle $c\in\gamma$. An affine fibration with obstruction cocycle $c\in\mathcal C^2(B,P)$ can be constructed inductively over skeleta (more precisely, we construct this fibration skeleton by skeleton, cell by cell).

Note that the fibration over each cell is uniquely defined. Indeed, let $\pi\colon E\to B$ be an affine fibration with the underlying fibration $T^*\!B/P$. The induced fibration $(\chi^k_j)^*E\to D^k$ over a cell $e^k_j$ is trivial (any locally trivial fibration over a disc is trivial). Every trivial fibration has a section. Thus the induced fibration $(\chi^k_j)^*E\to D^k$ is equivalent to  $(\chi^k_j)^*(T^*\!B/P)\to D^k$.

Given an affine fibration $\pi^{k-1}\colon E^{k-1}\to B^{k-1}$ with the underlying fibration $(T^*\!B/P)^{k-1}$, it remains to ``attach'' the fibration $(\chi^k_j)^*(T^*\!B/P)\to D^k$ over a cell $e^k_j$ to the given fibration $\pi^{k-1}\colon E^{k-1}\to B^{k-1}$. To do this, it suffices to find a section over the boundary a cell
$$
s\big|_{\partial e^k_j}\colon S^{k-1}\to(\chi^k_j\big|_{\partial e^k_j})^*E^{k-1}.
$$

There are four cases.

\textsl{Case $k=1$}. There is a section over $\partial e^1_j$, since the fibres are connected.

\textsl{Case $k=2$}. We must construct a fibration over the $2$-skeleton $B^2$ with obstruction cochain $c$.
Choose a point $b_j$ on each boundary $\partial e^2_j$. Then the cochain $c\in\mathcal C^2(B,P)$ assigns to each point $b_j$ an element of the lattice $p_j\in P_{b_j}$.

We must find a section
$$
s\big|_{\partial e^2_j}\colon S^2\to(\chi^2_j\big|_{\partial e^2_j})^*E^1
$$
that defines the element $p_j$ over the cell $e^2_j$ .

Now note that the fibration $\pi^1\colon E^1\to B^1$ is equivalent to $(T^*\!B/P)^1$, since the affine fibration $\pi^1\colon E^1\to B^1$ admits a section $s^1\colon B^1\to E^1$. And $(T^*\!B/P)^1$ is the quotient of $(T^*\!B)^1$ by $P^1$.

It is convenient to regard the circle $\partial e^2_j\simeq S^1$ as  a line segment $I=[b_j^1,b_j^2]$ with identified endpoints (these endpoints correspond to the point $b_j$ of the boundary $\partial e^2_j$). Denote by $(T^*\!B)^I$ the fibration over $I$ induced by the cotangent bundle $T^*\!B$. This fibration $(T^*\!B)^I\to I$ is trivial, therefore there exists a section $s_I\colon I\to(T^*\!B)^I$ such that $s_I(b_j^1)=0$, $s_I(b_j^2)=p_j$ (we want these two sections over endpoints to differ by $p_j$). The section $s_I\colon I\to(T^*\!B)^I$ induces the required section of $E^1\simeq(T^*\!B/P)^1$. (Take the quotient of $(T^*\!B)^I$ by the lattice $P$ and glue together endpoints of $I$ to get a section over the cell $e^2_j$.)

\textsl{Case $k=3$}. We must prove that the fibration $\pi^2\colon E^2\to B^2$ can be extended from $B^2$ to a cell $e^3_j$ (equivalently, we must prove that the fibration $(\chi^3_j\big|_{\partial e^3_j})^*E^2\to S^2$ has a section). To prove this, we will make use of the assumption that the obstruction cochain is a cocycle.

The case when the image of $\chi^3_j(\partial e^3_j)$ does not intersect the $1$-skeleton $B^1$ is clear. In this case the image $\chi^3_j(\partial e^3_j)$ is contained in an open cell $\operatorname{Int} e^2_m$. Any section over the cell $s\colon\operatorname{Int} e^2_m\to E^2$ (there exists a section $s$, since the fibration over the cell is trivial) defines a fibration over $\partial e^3_j$. We can assume without loss of generality that there exists a point $s_0$ of $\partial e^3_j$ such that $\chi^3_j(s_0)\in B^1$. It is convenient to regard the sphere $\partial e^3_j\simeq S^2$ with a fixed point $s_0$ as a unit disc with its boundary identified to a single point. In other words choose a map of pairs $g\colon(D^2,\partial D^2)\to(\partial e^3_j,s_0)$ which carries the interior of the disc $\operatorname{Int} D^2=D^2\setminus\partial D^2$ diffeomorphically onto $\partial e^3_j\setminus s_0$. Denote by $\chi\colon(D^2,\partial D^2)\to(B^2,B^1)$ the corresponding mapping of the disc to the $2$-skeleton.

Let us show that any section $s^1\colon B^1\to E^2$ can be extended from $s_0$ to the whole sphere $S^2$. Consider the following commutative diagram:

\begin{equation}
\label{eq4}
\begin{gathered}
\xymatrix{
&\pi_{2}(D^{2},\partial D^{2})\ar[r]^{\chi_{*}}\ar[d]_{h_1}&\pi_{2}(B^{2}, B^{1})\ar[d]_{h_2}
\\
H_{3}(e^{3}_{j}, \partial e^{3}_{j}; \mathbb{Z})\ar[r]^{\partial_{*}}
&H_{2}(\partial e^{3}_{j}, s_{0};
\mathbb{Z})\ar[r]^{(\chi^{3}_{j}\big|_{\partial e^{3}_{j}})_{*}}
& H_{2}(B^{2}, B^{1}; \mathbb{Z})\ar[r]^c
&\{\pi_{1}(F)\}
}
\end{gathered}
\end{equation}

Here $c$ is the obstruction cochain of the section $s^1$, the homomorphism $h_2$ is the Hurewicz homomorphism, and $h_1$ is the composition of the isomorphism  $g_*\colon \pi_2(D^2,\partial D^2)\to\pi_2(\partial e^3_j,s_0)$ with the Hurewicz homomorphism $\pi_2(\partial e^3_j,s_0)\to H_2(\partial e^3_j,s_0;\mathbb Z)$. The homomorphisms $h_1$ and $h_2$ are actually isomorphisms, since Hurewicz homomorphisms are isomorphisms. The section $s^1$ can be extended from $s_0$ to $\partial e^3_j$ if and only if the composition $\pi_2(D^2,\partial D^2)\to\{\pi_1(F)\}$ is a trivial homomorphism.
Consider the bottom row of the diagram \eqref{eq4}. It can be easily proved that the mapping
$$
\partial_*\circ(\chi^3_j\big|_{\partial e^3_j})_*\circ c\colon
H_3(e^3_j,\partial e^3_j;\mathbb Z)\to\{\pi_1(F)\}
$$
is precisely the restriction of the cellular cochain $\delta c\colon \mathcal C^3(B,\mathbb Z)\to\{\pi_1(F)\}$ to the subspace $H_3(e^3_j,\partial e^3_j;\mathbb Z)\hookrightarrow H^3(B^3,B^2;\mathbb Z)\simeq\mathcal
C^3(B,\mathbb Z)$. By construction $\delta c=0$, therefore the composition of homomorphisms in the bottom row is a trivial homomorphism. But $\partial_*$ is an isomorphism, therefore the mapping $\pi_2(D^2,\partial D^2)\to\{\pi_1(F)\}$ is trivial. Thus there exists a section over the sphere $\partial e^3_j\simeq S^2$.

\textsl{Case $k>3$}. In order to construct a section over $\partial e^k_j\simeq S^{k-1}$, we divide the sphere $S^{k-1}$ into to hemispheres $D^{k-1}_i$, $i=1,2$. There exists a section over each hemisphere (the fibration over each disc $D^{k-1}_i$ is trivial). The difference of these sections over the ``equator'' $D^{k-1}_1\cap D^{k-1}_2$
determines an element of $\pi_{k-2}(F)$. But this homotopy group is trivial ($F\simeq\mathbb T^m\times\mathbb
R^{n-m}$). Therefore, perturbing these sections in a neighbourhood of the ``equator'' $D^{k-1}_1\cap D^{k-1}_2$, we can glue the sections over hemispheres $D^{k-1}_i$ together.

Theorem~\ref{t8} is proved.

\end{proof}

\subsubsection{Classification in terms of transition functions}
\label{s2.4.2}

In this section we give another approach to the classification of affine fibrations $\pi\colon E\to B$ with the underlying fibration $\pi_0\colon V\to B$ (compare \cite[vol.~3, \S\,14]{7}).

Let $\{U_\alpha\}_{\alpha\in A}$ be a cover of $B$ such that there exists a section $s_\alpha\colon U_\alpha\to E$ over each element $U_\alpha$ of the cover (for instance, this holds if all $U_\alpha$ are contractible). We stress that the cover $\{U_\alpha\}_{\alpha\in A}$ is not necessarily open, that means that elements $U_\alpha$ may not be open. Then the section $s_\alpha$ determines an equivalence between the affine fibration $(\pi\colon E\to B)|_{U_\alpha}$ and the induced underlying fibration $(\pi_0\colon V\to B)\big |_{U_\alpha}$ over each section $U_\alpha$. It follows that affine fibrations $\pi\colon E \to B$ are classified by their set of transition functions $\gamma_{\alpha\beta}\colon U_\alpha\cap U_\beta\to V$ (where $\gamma_{\alpha\beta}\circ s_\alpha=s_\beta$).

Let us show that the set of transition functions $\{\gamma_{\alpha\beta}\}$ forms a $1$-cocycle for the nerve of the cover $N_{\{U_\alpha\}}$ with coefficients in $\mathcal V$. We need some definitions (to specify $\mathcal V$).

The \textit{nerve} $N_{\{U_\alpha\}}$ of a cover $\{U_\alpha\}_{\alpha\in A}$ is a simplicial complex such that there is one vertex  $v_\alpha\in N_{\{U_\alpha\}}$ for each element $U_\alpha$, and, moreover, there is an $m$-simplex
$\Delta^m_{\alpha_0,\dots,\alpha_m}$ with vertexes $v_{\alpha_i}$, $i=0,\dots,m$, whenever the intersection $\bigcap_{i=0}^mU_{\alpha_i}$ is nonempty.

Let us define a kind of simplicial cohomology. Given a simplicial complex $K$. Let us suppose that for any $m$-simplex $\Delta^m_{\alpha_0,\dots,\alpha_m}$ there corresponds a group $G^m_{\alpha_0,\dots,\alpha_m}$, and that for any pair of simplices $\Delta^m_{\alpha_0,\dots,\alpha_m}$ and $\Delta^{m+1}_{\alpha_0,\dots,\alpha_{m+1}}$ there is a coboundary operator
$$
\delta^m_{(\alpha_0,\dots,\alpha_m;\alpha_{m+1})}\colon G^m_{\alpha_0,\dots,\alpha_m}\to
G^{m+1}_{\alpha_0,\dots,\alpha_{m+1}}.
$$

Then simplicial homology $H^k(K,\{G^m_{\beta_0,\dots,\beta_m}\})$ with coefficients in $\{G^m_{\beta_0,\dots,\beta_m}\}$ are defined similarly to ordinary simplicial homology except that the coefficient at a simplex $\Delta^m_{\alpha_0,\dots,\alpha_m}$ is an element of $G^m_{\alpha_0,\dots,\alpha_m}$. In other words, the group of cochain $C^k(K,\{G^m_{\beta_0,\dots,\beta_m}\})$ consists of formal linear combinations
$$
\sum_{\alpha_0,\dots,\alpha_k}(g^k_{\alpha_0,\dots,\alpha_k}\Delta^k_{\alpha_0,\dots,\alpha_k}),
$$
where $g^k_{\alpha_0,\dots,\alpha_k}\in G^k_{\alpha_0,\dots,\alpha_k}$. The coboundary operator $\delta$ is defined by the formula:
\begin{align*}
&
\delta\biggl(\sum_{\alpha_0,\dots,\alpha_k}
(g^k_{\alpha_0,\dots,\alpha_k}\Delta^k_{\alpha_0,\dots,\alpha_k})\biggr)
\\
&\qquad=\sum_{\alpha_0,\dots,
\alpha_k}\biggl(\sum_{\alpha_{k+1}}\delta^k_{(\alpha_0,
\dots,\alpha_k;\alpha_{k+1})}(g^k_{\alpha_0,\dots,
\alpha_k})\Delta^{k+1}_{\alpha_0,\dots,\alpha_k,
\alpha_{k+1}}\biggr).
\end{align*}
Here $\Delta^m_{\sigma(\alpha_0,\dots,\alpha_m)}= (-1)^{\operatorname{sgn}\sigma}\Delta^m_{\alpha_0,\dots,\alpha_m}$ for each permutation $\sigma\colon\{\alpha_0,\dots,\alpha_m\}\to\{\alpha_0,\dots,\alpha_m\}$.

The coefficients $\mathcal V$ are defined as follows. Let $G^m_{\alpha_0,\dots,\alpha_m}$ be the group of section of the fibration $\pi_0\colon V\to B$ over the intersection $\bigcap_{i=0}^mU_{\alpha_i}$. Then $\mathcal V$ is the collection of these groups $\{G^m_{\alpha_0,\dots,\alpha_m}\}$. All coboundary operators $\delta^m_{\alpha_0,\dots,\alpha_m;\alpha_{m+1}}$ are defined by restrictions of sections.

Note that the set of transition functions $\{\gamma_{\alpha\beta}\}$ forms a $1$-cochain $c\in C^1(N_{\{U_\alpha\}},\mathcal V)$ which associates to each $1$-simplex of the nerve $N_{\{U_\alpha\}}$ the section of the fibration $\pi_0\colon V\to B$ over $U_\alpha\cap U_\beta$. In other words, the cochain $\gamma_{\alpha\beta}$ associates to each intersection $U_\alpha\cap U_\beta$ an element of $\mathcal V$.

\begin{lemma}
\label{l3}
{
The set of transition functions $\{\gamma_{\alpha\beta}\}$ forms a $1$-cocycle $c\in Z^1(N_{\{U_\alpha\}},\mathcal V)$ of the nerve $N_{\{U_\alpha\}}$ with coefficients in $\mathcal V$.}
\end{lemma}

\begin{proof}
We must prove that $\delta c=0$, that is $\gamma_{\alpha\beta}-\gamma_{\alpha\lambda}+\gamma_{\beta\lambda}=0$. It suffices to consider the action of these transition functions (regarded as sections of $V$) on $s_\alpha$:
$$
\gamma_{\lambda\alpha}\circ(\gamma_{\beta\lambda}\circ(
\gamma_{\alpha\beta}\circ s_\alpha))=s_\alpha.
$$

\end{proof}

\begin{theorem}
\label{t9} Let $U_\alpha$ be a cover of $B$ such that all $U_\alpha$ are contractible. Two affine fibrations $\pi_i\colon E_i\to B$, $i=1,2$, with the underlying fibration $\pi_0\colon V\to B$ are equivalent if and only if the corresponding elements of $H^1(N_{\{U_\alpha\}},\mathcal V)$ are equal.
\end{theorem}

\begin{proof}
There is a natural bijection between coboundaries and changes of sections over elements $\{U_\alpha\}$ of the cover. Thus for any fibration there corresponds a unique cohomology class, and two fibrations are equivalent if and only if they define equal cohomology classes.
\end{proof}

For any class $\gamma\in H^1(N_{\{U_\alpha\}},\mathcal V)$ and for any representative $c$ for this class there exists an affine fibration isomorphic to $(\pi_0\colon V\to B)\big|_{U_\alpha}$ over each $U_\alpha$ such that transition functions $\gamma_{\alpha\beta}$ are elements of the representative $c\in\gamma$.

All affine fibrations considered below are affine fibrations with the underlying fibration $T^*\!B/P$. Let us restate Theorem \ref{t9} for these fibrations.

Coefficients $C^\infty(T^*\!B/P)$ are defined analogously to $\mathcal V$ except that the group $G^m_{\alpha_0,\dots,\alpha_m}$ is the group of sections of the fibration $T^*\!B/P$ (that is latticed $1$-forms) over $\bigcap_{i=0}^mU_{\alpha_i}$.

\begin{corollary}
\label{c1}
 Let $U_\alpha$ be a cover of $B$ such that all $U_\alpha$ are contractible. Two affine fibrations $\pi_i\colon E_i\to B$, $i=1,2$, with the underlying fibration $T^*\!B/P$ are equivalent if and only if the corresponding elements of  $H^1(N_{\{U_\alpha\}},C^\infty(T^*\!B/P))$ are equal.
\end{corollary}

\subsubsection{Classification in terms of \v{C}ech cohomology}
\label{s2.4.3}

In this section we briefly describe the classification of Lagrangian fibrations in terms of \v{C}ech cohomology (sheaf cohomology). This classification was suggested by Duistermaat in \cite{2} and was used by Mishachev in \cite{4} (all necessary facts about sheafs are described in \cite[vol.~3, section~14]{7}).

Take the direct limit of the cohomology groups
$$
\varinjlim H^1(N_{\{U_\alpha\}},C^\infty(T^*\!B/P))
$$
over the system of all possible open covers $\{U_\alpha\}$, ordered by refinement. To take this limit, it is necessary to replace coefficients $C^\infty(T^*\!B/P)$ by the sheaf $\Gamma(T^*\!B/P)$ of sections of the quotient fibration $T^*\!B/P$.

Instead of taking the direct limit, one may use Leray's theorem which states that \v{C}ech cohomology are isomorphic to the cohomology for absolutely contractible covers defined as follows.

\begin{definition}
\label{d12} We say that a cover $\{U_\alpha\}$ of $B$ is \textit{absolutely contractible} if all $U_\alpha$ and all their finite intersections $\bigcap_{i=1}^NU_{\alpha_i}$ are either empty or contractible.
\end{definition}

Note that any Lagrangian fibration over an open contractible base $U_\alpha$ admits a Lagrangian section. The next theorem is analogous to Theorem \ref{t9}.

\begin{theorem}
\label{t10}
Let $P$ be a closed lattice on a manifold $B$. Two Lagrangian fibrations $\pi_i\colon (M_i,\omega_i)\to B$ with lattice $P$ are Lagrangian equivalent if and only if they determine the same element of $H^1(B,\Lambda(T^*\!B/P))$ (the same element of \v{C}ech $1$-cohomology with coefficients in the sheaf $\Lambda(T^*\!B/P)$ of Lagrangian sections of $T^*\!B/P$).
\end{theorem}

The advantage of \v{C}ech cohomology is that they are uniquely defined (they not depend on the choice of cover). On the other hand, the approach from section \ref{s2.4.2} can be applied to covers that are not necessarily open. For instance, an application of \v{C}ech cohomology to triangulations requires explanation.

\section{Invariants of Lagrangian fibrations}
\label{s3}

\subsubsection*{Classification theorems for Lagrangian and almost Lagrangian fibrations}

Combining results of \S\,\ref{s2} (see Lemmas \ref{l1} and \ref{l2}, Theorems \ref{t6}--\ref{t8}), we get the following theorem.

\begin{theorem}
\label{t11}
1) Any closed lattice $P$ on $B^n$ and any primary obstruction class $\gamma\in H^2(B^n,P)$ can be realized by an almost Lagrangian fibration, that is there exists an almost Lagrangian fibration $\pi\colon(M^{2n},\eta)\to B^n$ with lattice $P$ and obstruction $\gamma$.

2) The lattice $P$ and the obstruction $\gamma$ can be realized by a Lagrangian fibration if and only if corresponding almost Lagrangian fibrations have trivial symplectic structure obstruction.

3) For any two (almost) Lagrangian fibrations $\pi_i\colon(M^{2n}_i,\eta_i)\to B^n$, $i=1,2$, with the same lattices $P_i\subset T^*\!B$ and equal primary obstruction classes $\gamma_i\in H^2(B,P_i)$ there exists a $2$-form $\varphi$ on the base $B$ such that the twisted first fibration $\pi_1\colon(M^{2n}_1,\eta_1+\pi_1^*\varphi)\to B^n$ is Lagrangian equivalent to the second Lagrangian fibration $\pi_2\colon(M^{2n}_2,\eta_2)\to B^n$.

4) \,Two twisted (almost) Lagrangian fibrations $\pi\colon(M^{2n},\eta+\pi^*\varphi_i)\to B^n$, $i=1,2$, with lattice $P$ are Lagrangian equivalent if and only if $\varphi_1-\varphi_2=d\alpha$ for some latticed $1$-form $\alpha\in\Omega_P^1(B)$.
\end{theorem}

The next statement is simple.

\begin{lemma}
\label{l4} For any Lagrangian fibration $\pi\colon(M^{2n},\omega)\to B^n$ a fibration $\pi\colon(M^{2n},\omega+\pi^*\varphi)\to B^n$ is Lagrangian if and only if the twisting $2$-form $\varphi$ is closed.
\end{lemma}

We suggest the following plan of classification of almost Lagrangian fibrations over a base $B$ up to Lagrangian equivalence.
\renewcommand{\theenumi}{{\rm\arabic{enumi})}}
\renewcommand{\labelenumi}{\theenumi}
\begin{enumerate}
\item
Classify all closed lattices $P$ on $B$.
\item
For each closed lattice $P$ on $B$, classify affine fibrations with the underlying fibration $T^*\!B/P$. By Theorems \ref{t7} and \ref{t8}, it is sufficient to compute the group of primary obstruction classes $H^2(B,P)$.
\item
Compute the set of $2$-forms $\Omega^2(B,\mathbb R)$ modulo differentials of latticed $1$-forms $d\Omega_P^1(B)$.
\end{enumerate}

To classify Lagrangian fibrations, one should replace step 3) of the plan by the following steps $3')$ and $4')$.
\begin{itemize}
\item[$3')$]
Find (among fibrations from step 2)) all affine fibrations with the underlying fibration $T^*\!B/P$ which have trivial symplectic obstructions (see definition \ref{d6}).
\end{itemize}

After this step one would classify all Lagrangian fibrations up to a twisting. To classify twistings, it remains to compute the quotient $Z^2(B,\mathbb R)/d\Omega_P^1(B)$. Every $1$-form is a lattice $1$-form. Therefore, we can take the quotients of $Z^2(B,\mathbb R)$ and $d\Omega_P^1(B)$ by the subgroup of exact $2$-forms $d\Omega^1(B)$ (that is by differentials of $1$-forms). Denote by $H^2_P(B)$ the quotient space $d\Omega_P^1(B)/d\Omega^1(B)$ consisting of all cohomology classes in $H^2(B,\mathbb R)$ that can be realized by a differential of a latticed  $1$-form. The last step of our plan is the following.
\begin{itemize}
\item[$4')$] Compute \textit{nontrivial twistings} $H^2(B,\mathbb R)/H^2_P(B)$.
\end{itemize}

For some explicit computations of these invariants see sections \ref{s3.1}--\ref{s3.4}.

\begin{remark}
\label{r5} By Theorem ~\ref{t11}, almost Lagrangian fibrations over a base $B$ are classified up to a twisting by their  (closed) lattices $P\subset T^*\!B$ and primary obstruction classes $\gamma\in H^2(B,P)$. Twistings of an almost Lagrangian fibration are classified by the group
$$
\Omega^2(B,\mathbb R)/d\Omega_P^1(B).
$$

If the base $B$ is a two-manifold, then any almost Lagrangian fibration is Lagrangian. Thus their classifications coincide.
\end{remark}

\begin{example}
\label{ex4}
Classification of (almost) Lagrangian fibrations with fibre $\mathbb R^n$ (compare \cite{4}).

Any almost Lagrangian fibration $\pi\colon(M^{2n},\eta)\to B^n$ with fibre $\mathbb R^n$ is Lagrangian equivalent to a twisted cotangent bundle $\pi_0\colon(T^*\!B^n,\allowbreak \omega_0+\nobreak\pi_0^*\varphi)\to B^n$.

Two twisted cotangent bundles $\pi_0\colon(T^*\!B^n,\omega_0+\pi_0^*\varphi_i)\to B^n$, $i=1,2$, are Lagrangian equivalent if and only if the difference of the twisting forms is an exact form, that is \ $\varphi_2-\varphi_1=d\alpha$ for some $1$-form $\alpha$ on $B^n$.

A twisted Lagrangian fibration $\pi_0\colon(T^*\!B^n,\omega_0+\pi_0^*\varphi)\to B^n$ is Lagrangian if and only if the twisting $2$-form $\varphi$ is closed.
\end{example}

\subsubsection*{Classification in terms of sheafs}

Our plan of classification can be restated in terms of \v{C}ech cohomology (below we do not use this statement). We use notation of section \ref{s2.4.3}.


To classify Lagrangian fibrations, we must compute the group $H^1(B,\Lambda(T^*\!B/P))$). Consider the short exact sequence of sheafs
$$
\xymatrix{
0\ar[r]&\Lambda(T^*\!B/P)\ar[r]^i&\Gamma(T^*\!B/P)\ar[r]^>>>>>>d&Z^2B\ar[r]&0,}
$$
where $Z^2B$ is the sheaf of closed $2$-forms on $B$, the map $i$ is the natural inclusion, and $d$ is the exterior differential. Consider the associated exact cohomology sequence
\begin{align*}
&\xymatrix{
\dots \ar[r]^>>>>>{i_*}&H^0(B,\Gamma(T^*\!B/P))\ar[r]^>>>>>>{d_*}
&H^0(B,Z^2B)\ar[r]&H^1(B,\Lambda(T^*\!B/P))\ar[r]^>>>>>>{i_*}&{}}
\\
&\qquad\xymatrix{
{}\ar[r]^>>>>>>>{i_*}&H^1(B,\Gamma(T^*\!B/P))\ar[r]^>>>>>>{d_*}&H^1(B,Z^2B)\ar[r]&\dots
}.
\end{align*}

It turns out that our plan of classification is just a geometrical interpretation of this long exact sequence. Step 2) of the plan corresponds to the computation of the group $H^1(B,\Gamma(T^*\!B/P))$ (see section \ref{s2.4.2}). Step $3')$ corresponds to the computation of the kernel of the homomorphism $H^1(B,\Gamma(T^*\!B/P))\stackrel{d_*}\longrightarrow H^1(B,Z^2B)$ (check that $H^1(B,Z^2B)\simeq H^3(B,\mathbb R)$). The quotient space $H^0(B,Z^2B)/dH^0(B,\Gamma(T^*\!B/P))$ and nontrivial twistings $Z^2(B,\mathbb
R)/d\Omega_P^1(B)$ from section $4')$ are the same thing (the same group). To prove the last statement, use the following simple fact about \v{C}ech cohomology. Zero cohomology of a sheaf is the space of its global sections. Obviously, zero cohomology of $Z^2B$ and $\Gamma(T^*\!B/P)$ are spaces of $2$-form and latticed $1$-forms on $B$ respectively.

\subsection{Lattices}
\label{s3.1}

In this section we reduce the classification of closed lattices of rank $2$ on a surface $B$ up to isomorphism to the classification of some subgroups $G\subset\mathrm{GL}_2(\mathbb Z)\leftthreetimes\mathbb R^2$. (We need this result in order to classify compact Lagrangian fibrations over surfaces in \S\,\ref{s4}.)

We first define complete lattices (see definition \ref{d13}), then we prove that all lattices on compact surfaces are complete (see Theorem \ref{t13}, remark \ref{r6} and \S\,\ref{s4}), and finally we reduce the classification of complete lattices to the classification of some groups $G\subset\nobreak \mathrm{GL}_2(\mathbb Z)\leftthreetimes\nobreak\mathbb R^2$ (see corollaries \ref{c3} and \ref{c4}).

\subsubsection*{Complete lattices}

Let $L$ be a translation invariant lattice of rank $n$ on an $n$-dimensional affine space $\mathbb R^n$. Here $L$ is a lattice from definition \ref{d9}. Denote by $(\mathbb R^n,L)$ the affine space with the lattice $L$ on it and by $\mathrm{Aut}_n$ the automorphism group of $L$. There exists a basis of $\mathbb R^n$ with respect to which the lattice $L$ is the standard lattice $\mathbb Z^n$ and $\mathrm{Aut}_n\simeq \mathrm{GL}_n(\mathbb Z)\leftthreetimes\mathbb R^n\subset\mathrm{Aff}_n$. (Here $\mathrm{Aff}_n$ is the affine group of $\mathbb R^n$.)

Suppose $\widetilde{B}$ and $B$ are two manifolds, $f\colon\widetilde B\to B$ is a mapping and $P$ is a lattice on $B$. We say that the set $Q=f^*P=\bigcup_{\alpha\in P}f^*\alpha$ is the \emph{lattice on $\widetilde{B}$ induced by $f$} (if it is a lattice). If $f$ is a local diffeomorphism, in particular a covering map, then $Q$ is a lattice. Moreover, if $P$ is closed, then $Q$ is also closed.

A map $f\colon(B,P)\to(\mathbb R^n,L)$ is a \textit{developing map} if $f^*L=P$.

\begin{claim}
\label{cl6}
Any $n$-dimensional simply connected manifold $U^n$ with a closed lattice $Q$ of rank $n$ on it admits a developing map. This developing map is uniquely defined up to an automorphism of the lattice $L$ (that is up to an element of $\mathrm{Aut}_n$).
\end{claim}

\begin{proof}
Consider a point $x\in U^n$. Choose a basis of the lattice $Q_x$ at that point and extend it to $n$ covector fields $\alpha_1,\dots, \alpha_n\subset Q\subset\nobreak T^*U^n$. Integrals of these $1$-forms over paths define the developing map $f$. The map is well defined because the lattice is closed and $U^n$ is simply connected.

The map $f$ is the only developing map that takes point $x$ to the origin $O$ and takes the basis $\alpha_1(x),\dots,\alpha_n(x)$ of the lattice $Q$ at the point $x$ to the basis $dx^1,\dots,dx^n$ of the lattice $L$. Hence any other developing map $\widetilde{f}$ is the composition $g\circ f$ of the developing map $f$ with an automorphism $g$ of the lattice $L$.
\end{proof}

\begin{definition}
\label{d13}
A closed lattice $P$ of rank $n$ on an $n$-dimensional manifold $B^n$ is \textit{complete} if the developing map $f\colon(U^n,Q)\to(\mathbb R^n,L)$ of the universal cover $U^n$ with the induced lattice $Q$ is a diffeomorphism onto the whole $\mathbb R^n$.
\end{definition}

\subsubsection*{Affine manifolds}

Let us find all surfaces that admit a closed lattice of rank $2$ (see Theorem \ref{t12} and corollary \ref{c2} below) and prove that all lattices on these surfaces are complete (see Theorem \ref{t13} and remark \ref{r6}). We need some definitions.

\begin{definition}
\label{d14}
We say that a manifold $M^n$ admits an \textit{affine structure} if there exists an atlas $(U_\alpha,\varphi_\alpha)$ such that all transition maps $\varphi_\beta^{-1}\varphi_\gamma$ are affine transformations. An affine structure is called \textit{integral} if all transition functions $\varphi_\beta^{-1}\varphi_\gamma$ are affine transformations $x\to Ax+\vec b$ with an integer matrix $A$ (that is each function is an element of $\mathrm{GL}_n(\mathbb Z)\leftthreetimes\mathbb R^n\subset\mathrm{Aff}_n$). A manifold with an affine structure is an \textit{affine manifold}.
\end{definition}

An integral affine structure and a closed lattice of rank $n$ on a manifold $M^n$ are the same thing (see assertion \ref{cl6}).  A map $f\colon M^n_1\to M^m_2$ between two affine manifolds $M^n_1$ and $M^m_2$ is an \textit{affine map} if it is an affine transformation in affine coordinates (that is in chats of affine structures on $M^n_1$ and $M^m_2$). An affine map $f\colon M^n\to\mathbb R^n$ is a \textit{developing map}.

Assertion \ref{cl6} can be generalized as follows. For any simply connected affine manifold a developing map exists and it is unique up to an affine automorphism of  $\mathbb R^n$, that is up to an element of $\mathrm{Aff}_n$. In analogy with definition \ref{d13}, an affine structure on a manifold $B^n$ is \textit{complete} if the developing map $D\colon U^n\to\mathbb R^n$ of the universal cover $U^n$ with the induced lattice is a diffeomorphism onto the whole $\mathbb R^n$.

Recall that for any affine manifold  $B^n$ the action of the fundamental group $\pi_1(B)$ on the universal cover $U^n$ and the developing map $D\colon U^n\to\mathbb R^n$ define an affine holonomy $\rho\colon\pi_1(B)\to\mathrm{Aff}_n$ (given by the formula $\rho(g)\circ D=D\circ g$ for all $g\in\pi_1(B^n)$).

In \cite{8} J. Milnor obtained a result that can be restated as follows.

\begin{theorem}
\label{t12}
Among orientable two-dimensional surfaces only the torus $\mathbb T^2$ admits an affine structure.
\end{theorem}

\begin{corollary}
\label{c2}
A compact two-dimensional surface with a closed lattice of rank $2$ is either the torus $\mathbb T^2$ or the Klein bottle $\mathbb K^2$.
\end{corollary}

\begin{theorem}
\label{t13}
1) If $M^n$ is a manifold with a nilpotent fundamental group $\pi_1(M^n)$, then all closed lattices of rank $n$ on $M^n$ are complete.

2) Let $\widetilde  B^n$ be a covering space of $B^n$. If all closed lattices of rank $n$ on $\widetilde B^n$ are complete then all closed lattices of rank $n$ on $B^n$ are also complete.
\end{theorem}

\begin{proof}
1) It was proved in \cite{9} that if $M^n$ is a compact orientable affine manifold such that its affine holonomy is nilpotent (that is $\rho(\pi_1(M^n))$ is nilpotent) and preserves the volume (that is there exists an atlas whose transition maps $\varphi_j\varphi_i^{-1}$ belong to the group $\mathrm{SL}_n(\mathbb
R)\leftthreetimes\mathbb R^n$), then the affine structure on $M^n$ is complete.

In this case, affine holonomy is nilpotent because $\pi_1(M^n)$ is nilpotent. Transition maps $\varphi_j\varphi_i^{-1}$ of an integral affine manifold belong to the group $\mathrm{GL}_n(\mathbb
Z)\leftthreetimes\mathbb R^n$, and transition functions $\varphi_j\varphi_i^{-1}$ of an orientable manifold belong to $\mathrm{GL}_n^+(\mathbb R)\leftthreetimes\mathbb R^n$. Therefore $\varphi_j\varphi_i^{-1}\in
\mathrm{SL}_n(\mathbb R)\leftthreetimes\mathbb R^n$.

2)  Induce the lattice on  $\widetilde B^n$.
\end{proof}

\begin{remark}
\label{r6} Below (in \S\,\ref{s4}) a (compact) base $B$ is either the torus or the Klein bottle.  Both spaces satisfy conditions of \ref{t13}, therefore all closed lattices of rang $2$ on these spaces are complete.
\end{remark}

\subsubsection*{Classification of complete lattices}

Any manifold $B$ with a complete lattice $P$ is a quotient of the Euclidean space $(\mathbb R^n,L)$ with standard lattice $L$ by a free action of the fundamental group $\pi_1(B)$ (equivalently, the universal cover is isomorphic to $(\mathbb R^n,L)$). When do two action $\rho_i\colon\pi_1(B)\to\mathrm{Aut}_n$, $i=1,2$ define isomorphic lattices on $B$? The answer is given by Assertion \ref{cl8}. We start with a trivial assertion.

\begin{claim}
\label{cl7} Let $E$ be a manifold with lattice $Q$. Suppose that an action of a group $G$ preserves the lattice $Q$, and that the quotient map $\pi\colon E\to E/G$ is a covering map. Then there exists a lattice $P$ on the quotient space $E/G$ such that $Q=\pi^*P$. The lattice $Q$ is closed if and only if the lattice $P$ is closed.
\end{claim}

\begin{claim}
\label{cl8} Let $U$ be a simply connected manifold with lattice $Q$. Suppose that two actions $\rho_i\colon G_i\to\mathrm{Aut}_Q$ preserve the lattice $Q$, and that the quotient maps $\pi_i\colon U\to U/G_i$, $i=1,2$ are covering maps. The lattices $P_i$ on $U/G_i$ induced by the lattice $Q$ are isomorphic if and only if there exists an automorphism $f\in\mathrm{Aut}_Q$ which takes $\rho_1$ to $\rho_2$. That is there exists an isomorphism $\varphi\colon G_1\to G_2$ such that $f\circ(\rho_1(g))=\rho_2(\varphi(g))\circ f$ for all $g\in G_1$.
\end{claim}

\begin{proof}
($\Leftarrow$) The automorphism $f$ of the total space $(U,Q)$ induces an automorphism of quotient spaces.

($\Rightarrow$) Using homotopy lifting property, lift an automorphism of quotient spaces $(U/G_i,P_i)$ to an automorphism $f$ of the covering space $(U,Q)$ along paths. The automorphism $f$ is well defined, since the covering space is simply connected.

Let us define $\varphi$. Choose a point $x\in U$. Since $f$ is a fiberwise automorphism, for any element $g_1\in G_1$ there exists an element $g_2$ such that
\begin{equation}
f(\rho_{1}(g_{1})x) = \rho_{2}(g_{2})(f(x)).
\label{eq5}
\end{equation}
The element $g_2$ is uniquely defined, since the action of $\rho_2$ is free. Put $\varphi(g_1)=g_2$. Since  $\pi_2\colon U\to U/G_2$ is covering map, the element $g_2$ is the same for all elements in a neighbourhood of $x$. It yields that the element is the same for all points $x$ of $U$. Thus $\varphi$ is well defined.
\end{proof}

\begin{corollary}
\label{c3} There is a natural one-to-one correspondence between complete lattices (of rank $n$) on $n$-dimensional manifolds, considered up to isomorphism, and groups $G\subset\mathrm{Aut}_n$ such that the action of $G$ on $(\mathbb R^n,L)$ is free and the quotient space $\mathbb R^n/G$ is a manifold, considered up to conjugation by an element $g\in\mathrm{Aut}_n$.
\end{corollary}

\begin{corollary}
\label{c4} There is a natural bijection between closed complete lattices (of rank~$n$) on an $n$"=dimensional manifold $B$ and free actions $\rho\colon\pi_1(B)\to\mathrm{Aut}_n$ of the fundamental group $\pi_1(B)$ on $(\mathbb R^n,L)$ with quotient space $B$.
\end{corollary}

Note that
$$
\mathrm{Aut}_n\simeq \mathrm{GL}_n(\mathbb Z)\leftthreetimes\mathbb
R^n\subset\mathrm{Aff}_n
$$
This means that we reduced the classification of lattices to classification of matrix groups.

\subsection{Primary obstruction classes}
\label{s3.2}

Let $P$ be a closed lattice on $B$. To classify all affine fibrations $\pi\colon E\to B$ with the underlying fibration $\pi_0\colon T^*\!B/P\to B$, it suffices to group of primary obstruction classes $H^2(B,P)$ (see section \ref{s2.4}). This group is the quotient of the group of cocycles $\mathcal Z^2(B,P)$ by the subgroup of coboundaries $\mathcal
B^2(B,P)$. The group of cocycles $\mathcal Z^2(B,P)$ consists of all (obstruction) cochains $c\in\mathcal C^2(B,P)$ such that $\delta c=0$. The group of coboundaries $\mathcal B^2(B,P)$ consists of differential cochains
$c\in\mathcal C^2(B,P)$ assigned to sections over the $1$-skeleton $B^1$ (see Assertion \ref{cl5}).

In this section we calculate the group of primary obstruction classes in two special cases (see examples \ref{ex5} and \ref{ex6}). We need this result in \S\,\ref{s4}.

\begin{example}
\label{ex5}
The base $B$ is the quotient of the plane $\mathbb R^2$ by an action of a subgroup $\mathbb Z\oplus\mathbb Z\simeq\pi_1(\mathbb T^2)$ of the affine group $\mathrm{Aff}_2$. The images of generators of $\mathbb Z^2$ are $g=(\bigl(\begin{smallmatrix}1&n
\\
0&1\end{smallmatrix}\bigr),\bigl(\begin{smallmatrix}0
\\
y\end{smallmatrix}\bigr))$ and $h=(\bigl(\begin{smallmatrix}1&0
\\
0&1\end{smallmatrix}\bigr),\bigl(\begin{smallmatrix}x
\\
0\end{smallmatrix}\bigr))$, where $n\in\mathbb N$, $x,y\in\mathbb R$,
$x,y>0$. As a manifold, the base $B$ is diffeomorphic to the torus $\mathbb T^2$.

The lattice on $B$ is induced by the standard lattice $\mathbb Z^2$ on the plane $\mathbb R^2$. The induced lattice is well defined, since matrix parts of $g$ and $h$ are integer matrices.

Fix the following cell structure on the base. The base $B^2$ can be obtained from a parallelogram $OO_1O_2O_3$ with vertices $O=\bigl(\begin{smallmatrix}0
\\
0\end{smallmatrix}\bigr)$, $O_1=\bigl(\begin{smallmatrix}x
\\
0\end{smallmatrix}\bigr)$, $O_2=\bigl(\begin{smallmatrix}x
\\
y\end{smallmatrix}\bigr)$, $O_3=\bigl(\begin{smallmatrix}0
\\
y\end{smallmatrix}\bigr)$ by identifying opposite edges, in the way indicated by the arrows in the picture \ref{ri1},\,a.

\begin{figure}[h]
\begin{center}
\includegraphics{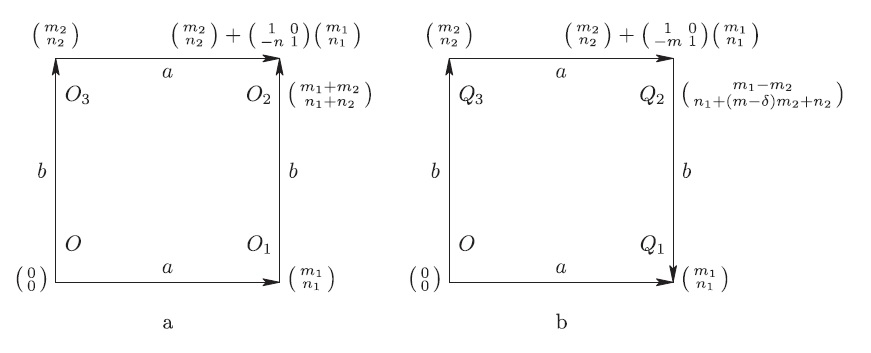}
\vskip3mm
\caption{Obstructions: torus (a) and Klein bottle(b)}
\label{ri1}
\end{center}
\end{figure}

Fix the following representation of obstruction cochains $c\in\mathcal C^2(B,P)$ by integer vectors. Identify every cochain $c\in\mathcal C^2(B,P)$ with the corresponding element of the lattice $P$ at the point $O=\bigl(\begin{smallmatrix}0
\\
0\end{smallmatrix}\bigr)$. Fix a basis of $P_O$.

\begin{claim}
\label{cl9} In this case
$$
\mathcal Z^2(\mathbb T^2,P)=\biggl\{\begin{pmatrix}
k
\\
l\end{pmatrix}\biggm|k,l\in\mathbb Z\biggr\},
\qquad
\mathcal B^2(\mathbb T^2,P)=\biggl\{\begin{pmatrix}0
\\
nm\end{pmatrix}\biggm|m\in\mathbb Z\biggr\}.
$$
Thus $H^2(\mathbb T^2,P)\simeq\mathbb Z\oplus\mathbb Z_n$.
\end{claim}

\begin{proof}
There is no $3$-cells, hence
$$
\mathcal Z^2(\mathbb T^2,P)=\mathcal C^2(\mathbb T^2,P)
=\biggl\{\begin{pmatrix}k
\\
l\end{pmatrix}\biggm|k,l\in\mathbb Z\biggr\},
$$
So, the calculation of cohomologies $H^2(\mathbb T^2,P)$ is reduced to the calculation of coboundaries $\mathcal B^2(\mathbb T^2,P)$.

In order to compute the group of coboundaries $\mathcal B^2(\mathbb T^2,P)$, we take two sections $s^1$, $\widetilde s^{\,1}$ over the $1$-skeleton $B^1$. To describe the difference of this sections we need the following notion. Let $\alpha$ be a latticed $1$-form on a path $s\colon I=[0,1]\to B$ from $x$ to $x$ (that is $\alpha\colon I\to
T^*\!B/P$, $\alpha(t)\in T^*_{s(t)}B/P$). Lift the latticed $1$-form, regarded as a section of $T^*\!B/P$, to a section $\widetilde\alpha\colon I\to T^*\!B$ over the path $s$. Then the difference $p=\widetilde\alpha(1)-\widetilde\alpha(0)\in P_x$ is an element of the lattice $P$ at point $x$. We say that the \textit{latticed $1$-form $\alpha$ changes by $p$ along $s$}.

The sections $s^1$, $\widetilde s^{\,1}$ differ by a latticed $1$-form $\alpha$ on the $1$-skeleton $B^1$. Suppose that the form $\alpha$ changes by $\bigl(\begin{smallmatrix}m_1
\\
n_1\end{smallmatrix}\bigr)$ along $a=OO_1$ and changes by
$\bigl(\begin{smallmatrix}m_2
\\
n_2\end{smallmatrix}\bigr)$ along $b=OO_3$.

To describe the difference of obstruction cochains for sections $s^1$ and $\widetilde s^{\,1}$, it is sufficient to find the change of $\alpha$ along $\partial e^2_i$. We compute the changes  $p_1,p_2\in P$ of $\alpha$ along $OO_1O_2$ and $OO_3O_2$ respectively and then we take the difference between $p_1$ and $p_2$.

As an example, we compute the change of the form along $O_3O_2$. Consider first the case $\bigl(\begin{smallmatrix}m_2
\\
n_2\end{smallmatrix}\bigr)=\bigl(\begin{smallmatrix}0
\\
0\end{smallmatrix}\bigr)$. Note that $O_3O_2$ is identified with $a=OO_1$ by means of
$g=(\bigl(\begin{smallmatrix}1&n
\\
0&1\end{smallmatrix}\bigr),\bigl(\begin{smallmatrix}0
\\
y\end{smallmatrix}\bigr))$. The affine transformation $g$ induces the linear transformation of the cotangent bundle $T^*\mathbb R^n\simeq\mathbb R^{2n}$ with matrix $(dg)^*=\bigr(\bigl(\begin{smallmatrix}1&n
\\
0&1\end{smallmatrix}\bigr)^{-1}\bigl)^T=\bigl(\begin{smallmatrix}1&0
\\
-n&1\end{smallmatrix}\bigl)$. If the form $\alpha$ changes by $\bigl(\begin{smallmatrix}m_1
\\
n_1\end{smallmatrix}\bigr)$ along  $OO_1$, then it changes by $\bigl(\begin{smallmatrix}m_1
\\
-m_1n+n_1\end{smallmatrix}\bigr)$ along $O_3O_2$.

If the sections $s^1$ and $\widetilde s^{\,1}$ differ by $\bigl(\begin{smallmatrix}m_2
\\
n_2\end{smallmatrix}\bigr)$ at point $O_3$, then the form $\alpha$ changes by $\bigl(\begin{smallmatrix}m_1+m_2
\\
-m_1n+n_1+n_2\end{smallmatrix}\bigr)$ along $OO_3O_2$. Other calculations are analogous (these calculations are readily seen in picture \ref{ri1},\,a.)

It follows that the difference cochain assigned to sections $s^1$ and $\widetilde s^{\,1}$, considered as an element of the lattice, is equal to $\bigl(\begin{smallmatrix}0
\\
nm_1\end{smallmatrix}\bigr)$.
\end{proof}
\end{example}

\begin{example}
\label{ex6} The base $B$ is again a quotient of the plane by an action of group $G\subset\mathrm{Aff}_2$. This time the base $B$ is the Klein bottle $\mathbb K^2$ and
$$
G=\biggl\langle
g=\biggl(\begin{pmatrix}1&m
\\
0&1\end{pmatrix},\begin{pmatrix}\alpha y
\\
y\end{pmatrix}\biggr),
\
h=\biggr(\begin{pmatrix}1&\delta
\\
0&-1\end{pmatrix},
\begin{pmatrix}x
\\
0\end{pmatrix}\biggr)\biggr\rangle,
$$
where $\alpha=\frac{m-\delta}2$, $m\in\mathbb Z$, $m\ge0$, $x,y\in\mathbb
R$, $x,y>0$, and $\delta$ is either $0$ or $1$ (this is a full list of closed lattices on the Klein bottle, see section \S\,\ref{s4}). The induced lattice $P$ on the Klein bottle $\mathbb K^2$ is a \textit{lattice from series $\mathbb K^2_{m,y;\delta,x}$}.

Just as in the previous example, the base $B^2$ can be obtained from a parallelogram $OQ_1Q_2Q_3$ with vertices $O=\bigl(\begin{smallmatrix}0
\\
0\end{smallmatrix}\bigr)$, $Q_1=\bigl(\begin{smallmatrix}x
\\
0\end{smallmatrix}\bigr)$, $Q_2=\bigl(\begin{smallmatrix}\alpha y+x
\\
y\end{smallmatrix}\bigr)$, $Q_3=\bigl(\begin{smallmatrix}\alpha y
\\
y\end{smallmatrix}\bigr)$
by identifying opposite edges, in the way indicated by the arrows in the picture ~\ref{ri1},\,b (in this case $Q_2Q_1=gh(OQ_3)$).

\begin{claim}
\label{cl10}
1) If $P$ is a lattice from series $\mathbb K^2_{m,y;0,x}$, then $H^2(\mathbb K^2,P)\simeq\mathbb Z_2\oplus\mathbb Z_m$. Moreover,
$$
\mathcal Z^2(\mathbb K^2,P)=\biggl\{\begin{pmatrix}k
\\
l\end{pmatrix}
\biggm|k,l\in\mathbb Z\biggr\},
\qquad
\mathcal B^2(\mathbb K^2,P)=\biggl\{\begin{pmatrix}2p
\\
mq\end{pmatrix}\biggm|p,q\in\mathbb Z\biggr\}.
$$

2)
If $P$ is a lattice from series $\mathbb K^2_{m,y;1,x}$, then $H^2(\mathbb K^2,P)\simeq\mathbb Z_{2m}$. Moreover,
$$
\mathcal Z^2(\mathbb K^2,P)=\biggl\{\begin{pmatrix}k
\\
l\end{pmatrix}\biggm|k,l\in\mathbb Z\biggr\},
\qquad
\mathcal B^2(\mathbb K^2,P)=\biggl\{\begin{pmatrix}2mp
\\
q\end{pmatrix}\biggm|p,q\in\mathbb Z\biggr\}.
$$
\end{claim}

\begin{proof}

In notation of example \ref{ex5} the difference cochain of two sections $s^1$, $\widetilde s^{\,1}$ is equal to $\bigl(\begin{smallmatrix}2m_2
\\
(\delta-m)m_2-mm_1\end{smallmatrix}\bigr)$ (see picture \ref{ri1},\,b).

In this case $Q_3Q_2=g(OQ_1)$ and $Q_2Q_1=gh(OQ_3)$. It follows easily that the form $\alpha$ changes by
$\bigl(\begin{smallmatrix}m_2
\\
n_2\end{smallmatrix}\bigr)+\bigl(\begin{smallmatrix}1&0
\\
-m&1\end{smallmatrix}\bigr)\bigl(\begin{smallmatrix}m_1
\\
n_1\end{smallmatrix}\bigr)$ along $OQ_3Q_2$ and it changes by
$\bigl(\begin{smallmatrix}m_1
\\
n_1\end{smallmatrix}\bigr)-\bigl(\begin{smallmatrix}1&0
\\
\delta-m&-1\end{smallmatrix}\bigr)\bigl(\begin{smallmatrix}m_2
\\
n_2\end{smallmatrix}\bigr)$ along $OQ_1Q_2$. Thus the subgroup $\mathcal B^2(\mathbb K^2,P)\subset\mathcal Z^2(\mathbb K^2,P)\simeq\mathbb Z\oplus\mathbb Z$ is generated by $(2,\delta-m)$ and $(0,-m)$. If $\delta=0$, then this subgroup is isomorphic to $2\mathbb Z\oplus m\mathbb Z\subset\mathbb Z\oplus\mathbb Z$. If $\delta=1$, then it is isomorphic to $2m\mathbb Z\oplus\mathbb Z\subset\mathbb Z\oplus\mathbb Z$.
\end{proof}
\end{example}

\subsection{Nontrivial twistings}
\label{s3.3}

Let $\pi\colon E\to B$ be an affine fibration with the underlying fibration $\pi_0\colon T^*\!B/P\to B$ for some closed lattice $P$ on $B$. Let us compute nontrivial twistings, that is the quotient of second cohomology group $H^2(B,\mathbb R)$ by the group of differentials of latticed $1$-forms $H^2_P(B)$. Below (to classify compact Lagrangian fibrations in section \S\,\ref{s4}) we need only the case when the base $B$ is the torus $\mathbb T^2$ (in other cases the group $H^2(B,\mathbb R)$ is trivial). Let us consider an example.

\begin{example}
\label{ex7}
Mishachev showed in \cite{4} that there are only two series of lattices (of rank $2$) on the torus $\mathbb T^2$. Both series are quotients of the plane $\mathbb R^2$ by an action of the group $G\subset
\mathrm{GL}_2(\mathbb Z)\leftthreetimes\mathbb
R^2\subset\mathrm{Aff}_2$, where $G\simeq\pi_1(\mathbb T^2)$ (that is $\mathbb Z\oplus\mathbb Z$ acts by affine transformations which preserves the standard lattice $\mathbb Z^2$ on the plane).

\textsl{Series $1$}. The generators of the group $G$ are
$$
g=\biggl(\begin{pmatrix}1&0
\\
0&1\end{pmatrix},\begin{pmatrix}u
\\
v\end{pmatrix}\biggr),
\qquad
h=\biggl(\begin{pmatrix}1&0
\\
0&1\end{pmatrix},\begin{pmatrix}w
\\
z\end{pmatrix}\biggr).
$$

\textsl{Series $2$}. The generators of the group $G$ are
$$
g=\biggl(\begin{pmatrix}1&0
\\
0&1\end{pmatrix},\begin{pmatrix}u
\\
0\end{pmatrix}\biggr),
\qquad
h=\biggl(\begin{pmatrix}1&n
\\
0&1\end{pmatrix},\begin{pmatrix}0
\\
z\end{pmatrix}\biggr).
$$
\begin{remark}
\label{r7} This is precisely series $\mathbb T^2_{u,v;w,z}$ and $\mathbb T^2_{n,z;u}$ from Theorem \ref{t15}.
\end{remark}

To compute nontrivial twistings $H^2(B,P)/H^2_P(B)$, it suffices to compute the cohomology class of the differential $d\alpha$ for each latticed $1$-form $\alpha$ on $B$. In the two-dimensional case it is sufficient to compute the integral $\displaystyle\int_{\mathbb T^2}\,d\alpha$.

To each latticed $1$-form $\alpha$ assign a $1$-form $\widetilde\alpha$ on the plane $\mathbb R^2$,
(take a pullback the latticed $1$-form $\alpha$, then choose any representative of the latticed $1$-form on the plane).

A $1$-form $\widetilde\alpha$ on the plane corresponds to a latticed $1$-form on the torus if and only if the action
$\pi_1(\mathbb T^2)$ preserves the form $\widetilde\alpha$ up to (global) sections of the lattice $\mathbb Z^2$.

\begin{claim}
\label{cl11} Let $\widetilde\alpha$ be a $1$-form such that
$$
(g^{-1})^*\widetilde\alpha=\widetilde\alpha+k\,dx+l\,dy,
\qquad
(h^{-1})^*\widetilde\alpha=\widetilde\alpha-p\,dx-q\,dy,
$$
where $k,l,p,q\in\mathbb Z$. Then for both series of lattices we have
$$
\int_{\mathbb T^2}\,d\alpha=pu+qv+kw+lz
$$
(here $v$ and $w$ are equal to $0$ for the second series).

Series 1. All $k,l,p,q$ are possible.

Series 2. The $1$-form $\widetilde\alpha$ can be realized by a latticed $1$-form $\alpha$ if and only if $k=0$.
\end{claim}

This statement was proved in other notations by Mishachev in \cite{4} (nontrivial twistings were called cohomological invariants of shifts). So we give only a

\textsc{Sketch of proof}.
The integral $\displaystyle\int_{\mathbb T^2}\,d\alpha$ is equal to the integral  of $d\widetilde\alpha$ over the parallelogram with vertices $\bigl(\begin{smallmatrix}0
\\
0\end{smallmatrix}\bigr)$, $\bigl(\begin{smallmatrix}u
\\
v\end{smallmatrix}\bigr)$, $\bigl(\begin{smallmatrix}w
\\
z\end{smallmatrix}\bigr)$ and $\bigl(\begin{smallmatrix}u+w
\\
v+z\end{smallmatrix}\bigr)$. Now the result follows from Stokes' theorem.

The only restriction on $k,l,p,q$ follows from $g^*h^*\widetilde\alpha=h^*g^*\widetilde\alpha$.
\end{example}

\subsection{Symplectic obstructions}
\label{s3.4}

Consider an almost Lagrangian fibration $\pi\colon(M^{2n},\eta)\to B^n$ over $B^n$. Let us show that the symplectic obstruction class $\gamma\in H^3(B,\mathbb R)$ (see definition \ref{d6}) is uniquely determined by the corresponding affine fibration over $3$-skeleton. (Below we do not use this result).

Recall that de Rham cohomology and singular cohomology are isomorphic.

\begin{theorem}[\rm{(De Rham's theorem)}]
Pairing of differential forms $\omega^k$ and chains $\sigma^k\colon\Delta^k\to B^n$, via integration
$$
\int_{\Delta^k}(\sigma^k)^*\omega^k,
$$
gives a homomorphism from $k$-forms to singular $k$-cochains
$$
\rho\colon\Omega^*(B^n,\mathbb R)\to C^*(B^n,\mathbb R)
$$

The homomorphism $\rho$ induces a canonical isomorphism (of algebras) between de Rham cohomology and singular cohomology.
\end{theorem}

In order to choose a representative cocycle $c^3_\gamma$ for $\gamma$, we assign to each $k$-simplex $\sigma^k\colon\Delta^k\to B^n$ (which we denote briefly by $[\sigma^k]$) a section $S(\sigma^k)\colon\Delta^k\to M^{2n}$ such that $\pi\circ S(\sigma^k)=\sigma^k$. (Actually we need only sections over $2$-dimensional and $3$-dimensional simplices.)

Since $\pi^*\gamma=[d\eta]$, the natural representative cocycle $c^3_\gamma$ for $\gamma$ assigns to each $3$-simplex  $[\sigma^3]$ the integral of the form $d\eta$ over the section $S(\sigma^3)$ over this $3$-simplex. But we need another representative for $\gamma$.

Consider the $2$-cocycle $c^2_\eta$ that assigns to each $2$-simplex $[\sigma^2]$ the integral of the form $\eta$ over the section $S(\sigma^2)$. Take the other representative cocycle $\widetilde c{\,}^3_\gamma=c^3_\gamma+\delta c^2_\eta$ for $\gamma$. Let us prove that the value of $\widetilde c{\,}^3_\gamma$ on each $3$-simplex depends only on transitions functions between sections over $2$-dimensional and $3$-dimensional simplices.

Suppose that $\partial[\sigma^3]=\sum_{i=0}^3(-1)^i[\sigma^2_i]$. Denote by $\alpha_i\colon\Delta^2\to T^*\!B/P$ the transition functions between sections $S(\sigma^3)$ and $S(\sigma^2_i)$ (note that $\pi_0\circ\alpha_i=\sigma^2_i$).

\begin{claim}
\label{cl12} The value of the cochain $\widetilde c{\,}^3_\gamma$ on a $3$-simplex $[\sigma^3]$ is equal to
$$
\sum_{i=0}^3(-1)^i\int_{\Delta^2}(d\alpha_i).
$$
\end{claim}

\begin{proof}
We have
\begin{align*}
&\int_{\Delta^{3}}\! (S(\sigma^{3}))^{*}\,d \eta- \sum_{i=0}^{3}
(-1)^{i} \int_{\Delta^{2}}\! (S(\sigma^{2}))^{*} \eta
=
\int_{\partial\Delta^{3}}\! (S(\sigma^{3}))^{*} \eta- \sum_{i=0}^{3}
(-1)^{i} \int_{\Delta^{2}}\! (S(\sigma^{2}))^{*} \eta
\\
&\qquad=\sum_{i=0}^{3} (-1)^{i} \int_{\Delta^{2}} (S(\sigma^{3}))^{*} \eta-
(S(\sigma^{2}))^{*} \eta= \sum_{i=0}^{3} (-1)^{i} \int_{\Delta^{2}}
(d\alpha_{i}).
\end{align*}
The last equality follows from Assertion \ref{cl3}.
\end{proof}

\begin{remark}
\label{r8}
If the space $B^n$ is triangulated, then one may consider only the simplices $[\sigma^3]$ of the triangulation (sometimes it is more convenient to consider $\gamma$ as an element of simplicial cohomology).
\end{remark}

\section{Classification of compact Lagrangian fibrations over surfaces}
\label{s4}

In this section we classify all Lagrangian fibrations $\pi\colon(M^4,\omega)\to B^2$ with compact total space $M^4$ (we say \textit{compact} Lagrangian fibrations for short). We follow the plan of classification from section \S\,\ref{s3}. In section \ref{s4.2} we classify all closed lattices $P$ of rank $2$ on surfaces $B^2$ (see Theorem \ref{t15}). In section \ref{s4.3} we compute primary obstructions $H^2(B^2,P)$ and nontrivial twistings $H^2(B^2,\mathbb R)/H^2_P(B^2)$ (see Theorem \ref{t16}).

In \cite{4} Mishachev classified all Lagrangian fibrations over orientable surfaces and showed that among non-orientable surfaces only the Klein bottle can be the base of a Lagrangian fibration (compare corollary \ref{c2}). Section \ref{s4.1} describes the fundamental group of the Klein bottle $\pi_1(\mathbb K^2)$.

\subsection{Fundamental group of the Klein bottle}
\label{s4.1}

The Klein bottle can be constructed from a quadrilateral by identification of sides (see picture \ref{ri2}). Recall that if $B$ is a CW-complex, then $1$-cells of $B$ can be identified with generators of the fundamental group, and $2$-cells of $B$ determine relations among those generators (this is Hurewicz isomorphism between homotopy groups and homology groups).

\begin{figure}[h]
\begin{center}
\includegraphics{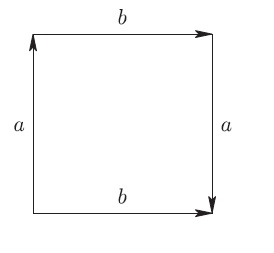}
\vskip3mm
\caption{Klein bottle}
\label{ri2}
\end{center}
\end{figure}

\begin{definition}
\label{d15} Two generators $a,b\in\pi_1(\mathbb
K^2)$ with the only relation
\begin{equation}
abab^{-1}=e.
\label{eq6}
\end{equation}
will be called \textit{standard generators}.
\end{definition}

For instance, sides $a$ and $b$ of the quadrilateral define standard generators.

\begin{claim}
\label{cl13}
Let $a,b\in\pi_1(\mathbb K^2)$ be standard generators. Then any element $x\in\pi_1(\mathbb K^2)$ can be uniquely written in the form $a^kb^l$. The correspondence $(k,l)\to x$ defines an isomorphism between $\pi_1(\mathbb K^2)$ and a semidirect product $\mathbb Z\leftthreetimes\mathbb Z$ with group operation $(k,l)(m,n)=(k+(-1)^lm,l+n)$.
\end{claim}

\begin{proof}
Relation \eqref{eq6} implies that $b^la^k=a^{(-1)^lk}b^l$, for all $k,l\in\mathbb Z$. Hence the correspondence $(k,l)\to x$ is a homomorphism. It remains to prove that the kernel is trivial, that is $a^kb^l\ne e$.

The the total degree of $b$ in \eqref{eq6} is $0$. Hence it suffices to prove that $a^k\ne e$. This follows from the following fact (see.\ \cite[Theorem~4.10]{10}) Let G be a group with generators $a_1, \dots , a_n$ and with a single defining relator $R(a_1, \dots, a_n)$. If $R(a_1, \dots, a_n)$ is a cyclically reduced word which involves $a_n$, then every (nontrivial) relator of $G$ must involve $a_n$. (A \textit{reduced} word is a word in which symbols $a_i$ and $a_i^{-1}$ do not occur consecutively. A \textit{cyclically reduced} is a reduced word which does not simultaneously begin with $a_i$ and end with $a_i^{-1}$.)
\end{proof}

Let us describe all standard generators.

\begin{claim}
\label{cl14} Let $a,b$ be standard generators of $\pi_1(\mathbb K^2)$. Then two other elements $a_1,b_1\in\pi_1(K^2)$ are standard generators if and only if $a_1=a^{\varepsilon_1}$, $b_1=a^kb^{\varepsilon_2}$ for some $\varepsilon_1=\pm1$, $\varepsilon_2=\pm1$, $k\in\mathbb Z$.
\end{claim}

\begin{proof}

($\Leftarrow$) Check that $a_1b_1a_1b_1^{-1}=e$ and that $a_1,b_1$ generate the whole group (express $a$ and $b$ in the form $a_1^pb_1^q$).

($\Rightarrow$) Let $a_1=(k_1,l_1)$, $b_1=(k_2,l_2)$. It follows from \eqref{eq6} that $l_1=0$ (more precisely, $2l_1=0$). Since any element can be written in the form $a_1^pb_1^q$
(in particular, standard generators $a$ and $b$ can be written in that form), it follows that $k_1=\pm1$ and $l_2=\pm1$.
\end{proof}
\begin{remark}
\label{r9} All standard generators $a_1$, $b_1$ can be described in algebraic terms.

Consider some standard generators $a$, $b$ of the group $\pi_1(\mathbb K^2)$.
Then the center $\pi_1(\mathbb K^2)$ ($\simeq\mathbb Z$) of the fundamental group is generated by $b^2$, and the commutator subgroup $(\pi_1(\mathbb K^2))'$ ($\simeq\mathbb Z$) is generated by $a^2$. Therefore an element $a_1^2$ of the fundamental group $\pi_1(\mathbb K^2)$ is a generator of the commutator subgroup if and only if $a_1=a^{\varepsilon_1}$. Analogously, an element $b_1^2$ is a generator of the center if and only if $b_1=a^kb^{\varepsilon_2}$
\end{remark}

Let us find the elements that preserve orientation. It is clearly seen in picture \ref{ri2} that the element $a\in\pi_1(\mathbb K^2)$ preserves orientation and the element $b$ reverses orientation. Assertion \ref{cl14} implies the following.

\begin{claim}
\label{cl15} For any standard generators $a,b$ the element $a$ preserves orientation and the element $b$ reverses orientation.
\end{claim}

Any element of $\pi_1(\mathbb K^2)$ has the form $a^kb^l$. Hence the following holds.

\begin{corollary}
\label{c5} The normal subgroup of $\pi_1(\mathbb K^2)$ consisting of orientation-preserving elements is generated by $a$ and $b^2$.
\end{corollary}

\subsection{Classification of complete lattices of rank 2}
\label{s4.2}

In section \ref{s3.1} we reduced the classification of $n$-dimensional lattices to the classification of subgroups $G\subset\mathrm{Aut}_n\simeq \mathrm{GL}_n(\mathbb Z)\leftthreetimes\mathbb R^n$ such that the quotient space is a manifold, up to an isomorphism $\mathrm{Aut}_n$.

Since the action is free, any element $g\in G\subset\mathrm{Aut}_n$ has an eigenvector, for any dimension $n$
(see Assertion \ref{cl16}). In dimension $2$ there exists a basis of the plane such that the linear part $dg$ of all $g\in G$ is an upper triangular matrix (see Lemma \ref{l6}). The rest of the proof is by exhaustion. The final result is summarized in Theorem \ref{t15}.

\begin{claim}
\label{cl16}
The action of a subgroup $G\subset\mathrm{Aff}_n\simeq \mathrm{GL}_n(\mathbb R)\leftthreetimes\mathbb R^n$ on $\mathbb R^n$ is free if and only if the equation $(A-E)x=-\vec b$ has no solutions for all $(A,\vec b)\in G$.
\end{claim}

By definition, the action is free if and only if $Ax+\vec b\ne x$ for all $x\in\mathbb R^n$ and for all $(A,\vec b)\in\nobreak G$

\begin{corollary}
The matrix part of each transformation has eigenvalue $1$.
\end{corollary}

\begin{prop}
\label{p1} Let $v=(v_1,\dots,v_n)$ be an $n$-dimensional vector $v=(v_1,\dots,v_n)$ with integer coeffitients. If the coefficients $v_i$ are relatively prime (that is the greatest common divisor of these elements is $1$), then the vector can be extended to an integer basis.
\end{prop}

The proposition can be easily proved by induction on $n$.
\smallskip

From now on we assume that the dimension $n$ is $2$.

\begin{claim}
\label{cl17} Let $g$ be an element of $G$. If $g$ preserves orientation, then there exists an integral basis of
$\mathbb R^2$ such that the matrix part $dg$ of the element is $\bigl(\begin{smallmatrix}1&n
\\
0&1\end{smallmatrix}\bigr)$, where $n\in\mathbb Z$, $n\ge0$. If $g$ reverses orientation, then there is an integral basis of the plane with respect to which the matrix $dg$ is $\bigl(\begin{smallmatrix}1&\delta
\\
0&-1\end{smallmatrix}\bigr)$, where $\delta$ is either $0$ or $1$. Both $n$ and $\delta$ are uniquely defined.
\end{claim}

\begin{proof}
Existence (of the basis). Choose an eigenvector $e_1$ of the operator $dg$ with eigenvalue $1$ and extend it to a basis. Then
$dg=\bigl(\begin{smallmatrix}1&m
\\
0&\pm1\end{smallmatrix}\bigr)$ with respect to the new basis $e_1,e_2$. If $m<0$ and $g$ preserves orientation, then invert $e_1$ to get $m>0$. If $g$ does not preserve orientation, then change $e_2$ to $e_2-\bigl[\frac m2\bigr]e_1$ (where $[x]$ is the biggest integer not exceeding $x$), to make $m$ equal either to $0$ or to $1$.

Uniqueness. Let $A=dg$ and $B=A-E$, where $E$ is the identity matrix. The upper-right element ($n$ or $\delta$) is uniquely determined by the greatest common divisor $N$ of elements of the matrix $B$. If $g$ preserves orientation, then the divisor is equal to $n$. If $\det(dg)=\det(A)=-1$, then the divisor $N$ and the element $\delta$ are either both even ($\delta=0$, $N=2$) or both odd ($N=\delta=1$). The divisor is an invariant, since the elements of $B$ in a new basis are linear combinations with integer coefficients of elements of $B$ in the given basis.
\end{proof}

Recall that an affine basis, that is $n+1$ points in general linear position for an $n$-dimensional affine space, can be considered as a point of the affine space together with a basis of the corresponding vector space.

\begin{definition}
An \textit{integral affine basis} is a point of an affine space $\mathbb R^n$ together with an integral basis of the corresponding vector space.
\end{definition}

\begin{lemma}
\label{l5} If $g\in G$ is not a translation, then there exists an integral affine basis of $\mathbb R^2$ such that $g$ has either the form $\bigl(\bigl(\begin{smallmatrix}1&n
\\
0&1\end{smallmatrix}\bigr),\bigl(\begin{smallmatrix}0
\\
y\end{smallmatrix}\bigr)\bigr)$, where $n\in\mathbb Z$, $n>0$,
$y\in\mathbb R\setminus0$, or the form
$\bigl(\bigl(\begin{smallmatrix}1&\delta
\\
0&-1\end{smallmatrix}\bigr),\bigl(\begin{smallmatrix}x
\\
0\end{smallmatrix}\bigr)\bigr)$, where $x\in\mathbb R\setminus0$,
and $\delta$ is either $0$ or $1$.
\end{lemma}

\begin{proof}
Choose a basis from Assertion \ref{cl17}. Suppose that the vector part of the element $g$ has the form
$\bigl(\begin{smallmatrix}z_1
\\
z_2\end{smallmatrix}\bigr)$. If $g$ preserves orientation them take $\bigl(\begin{smallmatrix}0
\\
-z_1/n\end{smallmatrix}\bigr)$ as the new origin. Otherwise take $\bigl(\begin{smallmatrix}0
\\
z_2/2\end{smallmatrix}\bigr)$.
\end{proof}

\begin{corollary}
\label{c7} The trace of $dg$ is $1+\det(dg)$.
\end{corollary}

\begin{lemma}
\label{l6}
There exists a basis of $\mathbb R^2$ such that the matrix part $dg$ of all elements $g\in G$ is upper triangular.
\end{lemma}

\begin{proof}
The case when all orientation-preserving elements of $G$ are translations is trivial. It suffices to make the matrix part of any orientation-reversing element upper triangular.

Let $g_1,g_2\in G$ be such that $dg_1=\bigl(\begin{smallmatrix}1&m
\\
0&1\end{smallmatrix}\bigr)$, $m\ne0$,
and $dg_2=\bigl(\begin{smallmatrix}a_{11}&a_{12}
\\
a_{21}&a_{22}\end{smallmatrix}\bigr)$. Then $\mathrm{tr}(d(g_1g_2))=a_{11}+a_{22}+ma_{21}$. By corollary \ref{c7}, we have $\mathrm{tr}(d(g_1g_2))=a_{11}+a_{22}$. We obtain $a_{21}=0$.
\end{proof}

Denote by $T$ those elements of $G$ which act as translation of the plane, and by $H$ those elements of $G$ which preserve the orientation of the plane. Obviously, $T\vartriangleleft H\vartriangleleft G$. The group $T$ is isomorphic either to $\mathbb Z$ or to $\mathbb Z^2$, since the orbits under the action $G$ are discrete. In this section we denote by $t,h,g$ elements of $T$, $H$ and $G$ respectively.

\begin{corollary}
\label{c8}
Either $H=T$ or $H/T\simeq\mathbb Z$.
\end{corollary}

\begin{proof}
There exists a homomorphism form $H$ to $\mathbb Z$ with kernel $T$. If the matrix part $dg$ of all elements $g\in G$ is upper triangular, then the homomorphism assigns to each element $h\in H$ the upper-right element of the matrix $dh$. \end{proof}

It follows that $G$ is generated by not more than four elements ($T$ is generated by not more than two elements, both
$H/T$ and $G/H$ are generated by not more than one element). If the quotient $\mathbb R^2/G$ is orientable, then $G$ is generated by not more than three elements. It implies that the quotient of the plane $\mathbb R^2$ by the action of $G$ must be one either $S^2$, $\mathbb{RP}^2$, $\mathbb T^2$, or $\mathbb K^2$ (combine the facts that the fundamental group of any other orientable surface is generated by more than four elements with only one relation and that every non-orientable surface has the orientable double cover)

Any closed lattice of rank $2$ on a surface determines an immersion of the surface in the plane $\mathbb R^2$ (developing map from section \ref{s3.1}). But the sphere can not be immersed in the plane, and the projective plane is covered by the sphere. Therefore neither the sphere nor the projective plane possesses a lattice (this also follows from corollary \ref{c2}).

\begin{lemma}
\label{l7}
Among compact two-dimensional surfaces only the torus and the Klein bottle admit a complete closed lattice of rank $2$.
\end{lemma}

Let us find all non-compact surfaces that admit a closed lattice of rank $2$ (see Lemma \ref{l9} below).

\begin{claim}
\label{cl18} If $H$ is not cyclic, then the quotient $\mathbb R^2/H$ is the torus.
\end{claim}

\begin{proof}
It suffices to prove that the quotient space $\mathbb R^2/H$ is compact. The case when $T$ is generated by two elements is clear. Without loss of generality we may assume that $H$ is generated by two elements $h$ and $t$.

First let us prove that $H$ is commutative. Choose the integral affine basis from Lemma \ref{l5} for the element $g$. If the translation $t$ has the form
$\bigl(\bigl(\begin{smallmatrix}1&0
\\
0&1\end{smallmatrix}\bigr),\bigl(\begin{smallmatrix}x
\\
z\end{smallmatrix}\bigr)\bigr)$, then the commutator $[g,t]=gtg^{-1}t^{-1}$ is $\bigl(\bigl(\begin{smallmatrix}1&0
\\
0&1\end{smallmatrix}\bigr),\bigl(\begin{smallmatrix}nz
\\
0\end{smallmatrix}\bigr)\bigr)$. The inequality $z\ne0$ would contradict the assumption that $T$ is generated by one element $t$.

The quotient space is compact because for any vector
$\bigl(\begin{smallmatrix}u
\\
v\end{smallmatrix}\bigr)\in\mathbb R^2$ there exists an element
$g^kt^l\in H$ such that $g^kt^l\bigl(\begin{smallmatrix}u
\\
v\end{smallmatrix}\bigr)=\bigl(\begin{smallmatrix}\widetilde u
\\
\widetilde v\end{smallmatrix}\bigr)$, where $0\le\widetilde u<x$
and $0\le\widetilde v<y$. Every element $g^kt^l$ has the form
$\bigl(\bigl(\begin{smallmatrix}1&kn
\\
0&1\end{smallmatrix}\bigr),\bigl(\begin{smallmatrix}lx
\\
ky\end{smallmatrix}\bigr)\bigr)$, consequently
$\bigl(\begin{smallmatrix}\widetilde u
\\
\widetilde v\end{smallmatrix}\bigr)=\bigl(\begin{smallmatrix}u+knv+lx
\\
v+ky\end{smallmatrix}\bigr)$. First choose proper $k$, then choose $l$.
\end{proof}

The remaining case is when $H$ is cyclic. We begin with simple algebra.

\begin{claim}
\label{cl19} Suppose that a group $G_2$ possesses an infinite cyclic subgroup $G_1$ of index $2$. Then $G_2$ is either commutative or isomorphic to the group of matrices
$$
\biggl\{\begin{pmatrix}\varepsilon&m
\\
0&1\end{pmatrix}\biggm|\varepsilon=\pm1,\ m\in\mathbb
Z\biggr\}.
$$
\end{claim}

This statement follows from the fact that there are only two automorphisms of the infinite cyclic group.

\begin{lemma}
\label{l8} The group $G$ is cyclic if and only if $H$ is cyclic.
\end{lemma}

\begin{proof}
Let $h$ be a generator of $H$ and let $g$ be an orientation-reversing element. Since $g^2\in H$, we have $g^2=h^l$. Then the square of the element $g'=gh^{-[l/2]}$ (where $[x]$ is largest integer not greater than $x$) is either $e$ or $h$ (by Assertion \ref{cl19} either $g$ and $h$ commute or $l=0$). But any affine transformation $f$ satisfying $f^2=e$ is a reflection. Every reflection has fixed points. It follows easily that $g'$ generates $G$.
\end{proof}

\begin{remark}
\label{r10} If $H$ is cyclic, then $G=H$ if and only if a generator of $G$ preserves orientation.
\end{remark}

\begin{lemma}
\label{l9} The plane, the cylinder and the M\"{o}bius strip are the only non-compact surfaces that admit a closed lattice of rank $2$.

The plane corresponds to the case $G=e$, the cylinder corresponds to the case when $G$ is generated by one
orientation-preserving element, and the M\"{o}bius strip corresponds to the case when $G$ is generated by one orientation-reserving element.
\end{lemma}

This Lemma follows directly from Lemmas \ref{l5} and \ref{l7}.

\begin{theorem}[\rm{(classification of complete lattices)}]
\label{t15}
1) Among compact two-dimensional surfaces only the torus $\mathbb T^2$ and the Klein bottle $\mathbb K^2$ admit a closed lattice of rank $2$.  Any closed lattice of rank $2$ on these surfaces is complete.

2) The plane $\mathbb R^2$, the cylinder $C^2$, the M\"{o}bius strip $\mathbb M^2$, the torus $\mathbb T^2$ and the Klein bottle $\mathbb K^2$ are the only surfaces that possess a complete lattice of rank $2$.

3) Table \ref{tab1} classifies all groups $G\subset\mathrm{Aut}_2$ which act freely on $\mathbb R^2$ so that the quotient space $\mathbb R^2/G$ is a manifold (classification is up to conjugation by an element of $\mathrm{Aut}_2$).

4) There is a natural bijection between groups $G$ from 3) and complete lattices on two-dimensional surfaces.
Each group $G$ defines a complete lattice on the quotient space $\mathbb R^2/G$, each complete lattice can be obtained in this way.

5) Lattices from different series are nonisomorphic.

6) Two lattices $\mathbb T^2_{u_1,v_1;w_1,z_1}$ and $\mathbb T^2_{u_2,v_2;w_2,z_2}$ are isomorphic if and only if there exist $C,D\in \mathrm{GL}(2,\mathbb Z)$ such that $\bigl(\begin{smallmatrix}u_2&w_2
\\
v_2&z_2\end{smallmatrix}\bigr)=C\bigl(\begin{smallmatrix}u_1&w_1
\\
v_1&z_1\end{smallmatrix}\bigr)D$.

7) Two lattices $C^2_{u_1,v_1}$ and $C^2_{u_2,v_2}$ are isomorphic if and only if there exists $C\in
\mathrm{GL}(2,\mathbb Z)$ such that $\bigl(\begin{smallmatrix}u_2
\\
v_2\end{smallmatrix}\bigr)=C\bigl(\begin{smallmatrix}u_1
\\
v_1\end{smallmatrix}\bigr)$.

8) Lattices from other series are pairwise nonisomorphic.
\end{theorem}
\begin{table}[h]
\renewcommand{\arraystretch}{1.2} 
\caption{Complete lattices of rank $2$} \label{tab1}
\vskip3mm
\begin{center}
\begin{tabular}{||c|c|c|c|c|c||} \hline
Base & Series & $g$ & $h$ & $t_1$ & $t_2$
\\
\hline\hline
$\mathbb R^2$ & $\mathbb R^2$ & -- & -- & -- & --
\\
\hline
$\mathbb T^1\times\mathbb R^1$ & $C^2_{u,v}$ & -- & -- &
$\bigl(\begin{smallmatrix}1&0
\\
0&1\end{smallmatrix}\bigr)\bigl(\begin{smallmatrix}u
\\
v\end{smallmatrix}\bigr)$ & --
\\
\hline
$\mathbb T^1\times\mathbb R^1$ & $C^2_{n,y}$ & -- &
$\bigl(\begin{smallmatrix}1&n
\\
0&1\end{smallmatrix}\bigr)\bigl(\begin{smallmatrix}0
\\
y\end{smallmatrix}\bigr)$ & -- & --
\\
\hline
$\mathbb M^2$ & $\mathbb M^2_{\delta,x}$ &
$\bigl(\begin{smallmatrix}1&\delta
\\
0&-1\end{smallmatrix}\bigr)\bigl(\begin{smallmatrix}x
\\
0\end{smallmatrix}\bigr)$ & -- & -- & $\bigl(\begin{smallmatrix}1&0
\\
0&1\end{smallmatrix}\bigr)\bigl(\begin{smallmatrix}2x
\\
0\end{smallmatrix}\bigr)$
\\
\hline
$\mathbb T^2$ & $\mathbb T^2_{u,v;w,z}$ & -- & -- &
$\bigl(\begin{smallmatrix}1&0
\\
0&1\end{smallmatrix}\bigr)\bigl(\begin{smallmatrix}u
\\
v\end{smallmatrix}\bigr)$ & $\bigl(\begin{smallmatrix}1&0
\\
0&1\end{smallmatrix}\bigr)\bigl(\begin{smallmatrix}w
\\
z\end{smallmatrix}\bigr)$
\\
\hline
$\mathbb T^2$ & $\mathbb T^2_{n,y;x}$ & -- &
$\bigl(\begin{smallmatrix}1&n
\\
0&1\end{smallmatrix}\bigr)\bigl(\begin{smallmatrix}0
\\
y\end{smallmatrix}\bigr)$ & $\bigl(\begin{smallmatrix}1&0
\\
0&1\end{smallmatrix}\bigr)\bigl(\begin{smallmatrix}x
\\
0\end{smallmatrix}\bigr)$ & --
\\
\hline
$\mathbb K^2$ & $\mathbb K^2_{0,y;\delta,x}$ &
$\bigl(\begin{smallmatrix}1&\delta
\\
0&-1\end{smallmatrix}\bigr)\bigl(\begin{smallmatrix}x
\\
0\end{smallmatrix}\bigr)$ & -- & $\bigl(\begin{smallmatrix}1&0
\\
0&1\end{smallmatrix}\bigr)\bigl(\begin{smallmatrix}-\frac{\delta y}2
\\
y\end{smallmatrix}\bigr)$ & $\bigl(\begin{smallmatrix}1&0
\\
0&1\end{smallmatrix}\bigr)\bigl(\begin{smallmatrix}2x
\\
0\end{smallmatrix}\bigr)$
\\
\hline
$\mathbb K^2$ & $\mathbb K^2_{n,y;0,x}$ &$\bigl(\begin{smallmatrix}1&0
\\
0&-1\end{smallmatrix}\bigr)\bigl(\begin{smallmatrix}x
\\
0\end{smallmatrix}\bigr)$ &$\bigl(\begin{smallmatrix}1&n
\\
0&1\end{smallmatrix}\bigr)\bigl(\begin{smallmatrix}\frac{ny}2
\\
y\end{smallmatrix}\bigr)$ & -- & $\bigl(\begin{smallmatrix}1&0
\\
0&1\end{smallmatrix}\bigr)\bigl(\begin{smallmatrix}2x
\\
0\end{smallmatrix}\bigr)$
\\
\hline
$\mathbb K^2$ & $\mathbb K^2_{2n,y;1,x}$ & $\bigl(\begin{smallmatrix}1&1
\\
0&-1\end{smallmatrix}\bigr)\bigl(\begin{smallmatrix}x
\\
0\end{smallmatrix}\bigr)$ & $\bigl(\begin{smallmatrix}1&2n
\\
0&1\end{smallmatrix}\bigr)\bigl(\begin{smallmatrix}\frac{(2n-1)y}2
\\
y\end{smallmatrix}\bigr)$ & -- & $\bigl(\begin{smallmatrix}1&0
\\
0&1\end{smallmatrix}\bigr)\bigl(\begin{smallmatrix}2x
\\
0\end{smallmatrix}\bigr)$
\\
\hline
\end{tabular}
\end{center}
\end{table}

\subsubsection*{Description of table \ref{tab1}}

Each row of Table \ref{tab1} contains a base (a surface), the notion for a series of lattices on the base, and
generators of the fundamental group of the base (considered as elements of $\mathrm{Aut}_2$).

In the table we denote by $g$ an element that does not preserve orientation and by $h$ an orientation-preserving  element which is not a translation. We denote translations by $t_1$ and $t_2$. For orientable surfaces, the elements in the table form a basis of the fundamental group. In the case of the Klein bottle, elements in the three first columns  (that is $h,g$ or $t_1,g$) are standard generators.

For non-orientable surfaces, we put $t_2=g^2$. This means that elements in the three last columns of the table generate the subgroup $H\subset G$ consisting of all elements that preserve orientation.

In the table $x,y,u,v,w,z\in\mathbb R$, $x,y>0$, $n\in\mathbb N$, and $\delta$ is either $0$ or $1$.

\begin{proof}[of Theorem \ref{t15}]
We have already proved 1),~2) and~4) (combine corollaries \ref{c2},~\ref{c3}, remark~\ref{r6}
and Lemmas \ref{l7},~\ref{l9}). It remains to prove the following: first,
any group $G$ from item 3) can be reduced to the required form; secondly, the action of each such group is free (we must check that the lattice on the quotient space $\mathbb R^2/G$ is well defined); and thirdly, series from Table \ref{tab1} are pairwise nonisomorphic (more precisely, we must prove items 5)--8)).

\textsl{Reduction to the required form}. We make use of Lemmas \ref{l5} and \ref{l6}. After that it is clear how to proceed further (how to change the basis).

Consider the case of the Klein bottle $B=\mathbb K^2$ (other cases are analogous).

Take standard generators $a,b\in\pi_1(\mathbb K^2)$ (that is generators with the only relation $ab=ba^{-1}$). By Lemma \ref{l6} there exists an integral affine basis such that
$a=\bigl(\bigl(\begin{smallmatrix}1&n
\\
0&1\end{smallmatrix}\bigr),\bigl(\begin{smallmatrix}x_1
\\
x_2\end{smallmatrix}\bigr)\bigr)$, where $n\in\mathbb Z$,
$x_1,x_2\in\mathbb R$, and the element $b$ has the form
$\bigl(\bigl(\begin{smallmatrix}\sigma&m
\\
0&-\sigma\end{smallmatrix}\bigr),\bigl(\begin{smallmatrix}y_1
\\
y_2\end{smallmatrix}\bigr)\bigr)$, where $m\in\mathbb Z$, $\sigma=\pm1$,
$y_1,y_2\in\mathbb R$ (by Assertion \ref{cl15}, the element $a$ preserves orientation).

We must reduce the standard generators $a$ and $b$ to the required form by the following operations. We can choose an integral affine basis of $\mathbb R^2$ and we can change standard generators of the Klein bottle $\mathbb K^2$. To simplify notations throughout this part of the proof we use the same notations for parameters of $a$ and $b$ (for example, $n$ is always the upper-right element of $da$).

Without loss of generality it can be assumed that $\sigma=1$.
It is not hard to check that the identity $ab=ba^{-1}$ holds if and only if $x_1+y_1+ny_2=y_1-\sigma x_1+\sigma nx_2-mx_2$ and $x_2+y_2=y_2+\sigma x_2$. If $\sigma=-1$, then $x_2=0$ and, consequently, $n=0$ (the action of $a$ is free). If $a$ is a translation, then, by Lemma \ref{l6}, there exists an integral affine basis of $\mathbb R^2$ such that $\sigma=1$.

Let us reduce the linear parts $da$ and $db$ to the required form. If $m$ and $n$ are odd, then substitute $ba^{-1}$ for $b$ (recall assertion~\ref{cl14}) to change the parity of $m$. If $n<0$, then invert the first basis vector $e_1$ to get $n\ge0$. Further, change the second basis vector ~$e_2$ to $e_2-\bigl[\frac m2\bigr]e_1$ to make $m$ equal either to $0$ or to $1$) (compare \ref{cl17}).

Now, we reduce the vector parts of the elements $a$ and $b$ to the required form. Take
$\bigl(\begin{smallmatrix}0
\\
-z_1/n\end{smallmatrix}\bigr)$ as the new origin to make $y_2$ equal to $0$ (compare Lemma \ref{l5}). Further, if $x_2<0$, replace $e_1$, $e_2$ by $-e_1$, $-e_2$ to get $x_2>0$. After that, if we have $y_1<0$, then substitute $b^{-1}$ for $b$ to change the sign of $y_1$. Remaining restrictions on $a$ and $b$ (that is $x_1=\frac{n-\delta}2x_2$) follows from the identity $ab=ba^{-1}$.

\textsl{The action is free}. Consider a group $G$ from Table \ref{tab1}. By Assertion \ref{cl16}, we must prove that for all elements $(A,\vec b)\in G$ the equation $(A-E)x=-\vec b$ has no solutions. The proof is straightforward. To simplify the case of the Klein bottle we describe all element of the group.

\begin{lemma}
\label{l10} If $a=\bigl(\bigl(\begin{smallmatrix}1&n
\\
0&1\end{smallmatrix}\bigr),\bigl(\begin{smallmatrix}\alpha x
\\
x\end{smallmatrix}\bigr)\bigr)$
and $b=\bigl(\bigl(\begin{smallmatrix}1&\delta
\\
0&-1\end{smallmatrix}\bigr),\bigl(\begin{smallmatrix}y
\\
0\end{smallmatrix}\bigr)\bigr)$, then
\begin{equation}
a^{k}b^{l}=\left(\begin{pmatrix} 1 &
\dfrac{1+(-1)^{l-1}}{2}\delta+(-1)^{l}kn
\\
0 & (-1)^{l}\end{pmatrix},
\\
\begin{pmatrix} k \alpha x+ \dfrac{k(k-1)}{2}nx +ly
\\
kx \end{pmatrix}\right).
\end{equation}
\end{lemma}

The proof is by direct calculation.
\smallskip

\textsl{Series are nonisosmorphic}. Lattices on non-diffeomorphic surfaces are pairwise nonisomorphic. For orientable surfaces, lattices from different series are pairwise nonisomorphic (one series is generated by translations and another is not generated by translations).

In all other cases we make use of Assertion \ref{cl8} (about isomorphic lattices on the quotient space). Items 6),~7) are clear. The matrix $C$ corresponds to changes of basis of the plane (changes of the origin do not affect translations), ant the matrix $D$ corresponds to changes of basis of the fundamental group. Notice that different generators of the fundamental group of the cylinder define different signs in the formula from item 7). This can be ignored, since the matrix $-E$ is an element of the group $\mathrm{GL}_2(\mathbb Z)$.

The case $\mathbb T^2_{n,y;x}$ was considered by Mishachev in \cite{4} (this case can also be considered analogously to the case of the Klein bottle below). Consider the Klein bottle $\mathbb K^2$. Next lemma describes all automorphisms between two lattices on the Klein bottle $\mathbb K^2$.

\begin{lemma}
\label{l11} Suppose that two lattices $P_i$, $i=1,2$, on the Klein bottle are defined by two pairs of standard generators $a_i,b_i$, where $a_i\,{=}\,\bigl(\bigl(\begin{smallmatrix}1&n_i
\\
0&1\end{smallmatrix}\bigr),\allowbreak \bigl(\begin{smallmatrix}\alpha_iy_i
\\
y_i\end{smallmatrix}\bigr)\bigr)$ and
$b_i=\bigl(\bigl(\begin{smallmatrix}1&\delta_i
\\
0&-1\end{smallmatrix}\bigr),\bigl(\begin{smallmatrix}x_i
\\
0\end{smallmatrix}\bigr)\bigr)$.

An automorphism $f$ of the plane takes $P_1$ to $P_2$ if and only if
\begin{align}
fa_{1}f^{-1}&=a_{2}^{\varepsilon_{1}},
\label{eq8}
\\
fb_{1}f^{-1}&=a_{2}^{k}b_{2}^{\varepsilon_{2}}.
\label{eq9}
\end{align}
for some $\varepsilon_1=\pm1$, $\varepsilon_2=\pm1$ and $k\in\mathbb Z$.
If $a_1\ne a_2$ or $b_1\ne b_2$, then there is no such automorphism $f$ (lattices from different series are nonisomorphic). If $a_1=a_2$ and $b_1=b_2$, then $f$ is an automorphism if and only if the following holds:
\begin{itemize}
\item[1)]
$f=\bigl(\bigl(\begin{smallmatrix}\varepsilon_2&\mu
\\
0&\varepsilon_1\end{smallmatrix}\bigr),\bigl(\begin{smallmatrix}z_1
\\
\frac{ky}2\end{smallmatrix}\bigr)\bigr)$, where
$\mu=\frac{-\varepsilon_1\delta+\varepsilon_1kn+\varepsilon_2\delta}2$,
$z_1\in\mathbb R$;
\item[2)] if $n\ne0$, then $\varepsilon_2=1$;
\item[3)] if $n$ is odd, then $k$ is even;
\item[4)] if $\delta\ne0$ (and $n$ is even), then $\varepsilon_1=\varepsilon_2$.
\end{itemize}
\end{lemma}
\begin{proof}
It follows from Assertions \ref{cl8} and \ref{cl14} that $f$ is an automorphisms if and only if identities \eqref{eq8} and \eqref{eq9} hold.

Note that any automorphisms $f$ preserves the first basis vector $e_1$, since $e_1$ is the only common eigenvector with eigenvalue $1$ for all elements of $\pi_1(\mathbb K^2)$. The next assertion describes the action of an automorphism.

\begin{claim}
\label{cl20} Let $e=\bigl(\bigl(\begin{smallmatrix}1&\widetilde m
\\
0&\sigma\end{smallmatrix}\bigr),\bigl(\begin{smallmatrix}\widetilde x
\\
\widetilde y\end{smallmatrix}\bigr)\bigr)$, where $\widetilde
m,\widetilde x,\widetilde y,\sigma\in\mathbb R$. Then the conjugation by an automorphism
$$
f=\biggl(\begin{pmatrix}\sigma_1&\lambda
\\
0&\sigma_2\end{pmatrix},\begin{pmatrix}z_1
\\
z_2\end{pmatrix}\biggr),
$$
where $\sigma_1=\pm1$, $\sigma_2=\pm1$, $\lambda\in{\mathbb Z}$, $z_1,z_2\in\mathbb R$,
is given by the formula
\begin{equation}
\label{eq10} fef^{-1}=\biggl(\begin{pmatrix} 1 & (\sigma-1)
\sigma_{2} \lambda+ \sigma_{1} \sigma_{2} \widetilde{m}
\\
0 & \sigma\end{pmatrix},
\begin{pmatrix} \sigma_{1}\widetilde{x} + \lambda\widetilde{y}
+ (1-\sigma)\sigma_{2} \lambda z_{2} - \sigma_{1} \sigma_{2}
\widetilde{m} z_{2}
\\
(1- \sigma)z_{2} + \sigma_{2} \widetilde{y}
\end{pmatrix}\biggr).
\end{equation}
\end{claim}

In particular, for elements $a$ and $b$ from Lemma \ref{l11} we have:

\begin{align}
\label{eq11}
faf^{-1}&= \Biggl(\begin{pmatrix} 1 &
\sigma_{1}\sigma_{2}n
\\
0 & 1 \end{pmatrix},
\begin{pmatrix} \sigma_{1}\dfrac{n-\delta}{2}y
-\sigma_{1}\sigma_{2}nz_{2}+\lambda y
\\
\sigma_{2}y \end{pmatrix}\Biggr),
\\
fbf^{-1}&=\biggl(\begin{pmatrix} 1 &
-2\sigma_{2}\lambda+\sigma_{1}\sigma_{2}\delta
\\
0 & -1 \end{pmatrix},
\begin{pmatrix} \sigma_{1}x+
2\sigma_{2}\lambda z_{2}-\delta\sigma_{1} \sigma_{2}z_{2}
\\
2z_{2} \end{pmatrix}\biggr). \label{eq12}
\end{align}

Hence equations \eqref{eq8} and \eqref{eq9} are equivalent to the following:

\begin{gather}
\sigma_{1}\sigma_{2}n_{1} = \varepsilon_{1}n_{2},
\label{eq13}
\\
-2\sigma_{2}\lambda+\sigma_{1}\sigma_{2}\delta_{1}=
\delta_{2} - kn_{2},
\label{eq14}
\\
\sigma_{1}\frac{n_{1}-\delta_{1}}{2}y_{1}
-\sigma_{1}\sigma_{2}n_{1}z_{2}+\lambda y_{1} = \frac{n_{2} -
\varepsilon_{1} \delta_{2}}{2} y_{2},
\label{eq15}
\\
\sigma_{2}y_{1} = \varepsilon_{1} y_{2},
\label{eq16}
\\
\sigma_{1}x_{1}+ 2\sigma_{2}\lambda z_{2}-\delta_{1} \sigma_{1}
\sigma_{2}z_{2} = k \frac{n_{2} -\delta_{2}}{2} y_{2} +
\frac{k(k-1)}{2}n_{2}y_{2} + \varepsilon_{2}x_{2},
\label{eq17}
\\
2z_{2} = k y_{2}.
\label{eq18}
\end{gather}

Compare matrix parts of $a_1$, $a_2$ and $b_1$, $b_2$ to obtain \eqref{eq13} and \eqref{eq14}.
To get \eqref{eq15} and \eqref{eq16} compare vector parts of $a_1$ and $a_2$. To get \eqref{eq17} and \eqref{eq18} compare vector parts of $b_1$ and $b_2$.

Since $y_i>0$, $\sigma_i=\pm1$ and $\varepsilon_i=\pm1$, it follows from identity \eqref{eq16} that $\sigma_2=\varepsilon_1$ and $y_1=y_2$. Analogously, it follows from identity \eqref{eq13} that $n_1=n_2$. Moreover, if $n_i\ne0$, then $\sigma_1=1$.

From \eqref{eq14} and \eqref{eq18} it follows that
$$
\lambda=\frac{-\sigma_2\delta_2+k\sigma_2n_2+\sigma_1\delta_1}2,
\qquad
z_2=\frac{ky_2}2.
$$ Combining obtained identities with \eqref{eq17}, we get $y_1=y_2$, $\sigma_1=\varepsilon_2$.

It remains to prove that $\delta_1=\delta_2$ (then we get $a_1=a_2$ and $b_1=b_2$). If $n_1$ ($=n_2$) is odd, then $\delta_1=\delta_2=0$ (and $k$ is even, since $\lambda$ is an integer). If $n_1$ ($=n_2$) is even, then $\sigma_2\delta_2=\sigma_1\delta_1$, since $\lambda$ is an integer (that is $\delta_1=\delta_2$ and if $\delta_i\ne0$, then $\sigma_1=\sigma_2$).

Identity \eqref{eq15} is automatically satisfied.
\end{proof}

Theorem \ref{t15} is proved.
\end{proof}

\begin{remark}
\label{r11}
For any complete lattice on the Klein bottle (see Theorem \ref{t15}) the corresponding lattice on the torus
(that is on the orientable double cover) is defined by the element in the three right rows of Table \ref{tab1} (the torus with the lattice if the quotient of $(\mathbb R^2,\mathbb Z^2)$ by the action of the group generated by these elements).
\end{remark}

\subsection{Other invariants}
\label{s4.3}

Consider a complete lattice $P$ on a two-dimensional manifold $B^2$. Let us compute the group of primary obstructions $H^2(B^2,P)$ and nontrivial twistings $H^2(B^2,\mathbb R)/H^2_P(B^2)$. (Actually we have already computed them in \S\,\ref{s3}.)

\begin{theorem}[\rm{(invariants of Lagrangian fibrations)}]
\label{t16}
1)  If $B^2$ is neither the torus nor the Klein bottle, then both $H^2(B^2,P)$ and $H^2(B^2,\allowbreak\mathbb R)/H^2_P(B)$ are trivial for any complete lattice $P$ on $B^2$.

2) For any complete lattice $P$ from series $\mathbb T^2_{u,v;w,z}$ the space $H^2(\mathbb T^2,P)$ is isomorphic to $\mathbb Z\oplus\mathbb Z$. The group $H^2_P(\mathbb T^2)$ is the subgroup of $H^2(\mathbb T^2,\mathbb R)\simeq\mathbb R$ spanned by $u,v,w,z$.

3) For any complete lattice $P$ from series $\mathbb T^2_{n,y;x}$ the space $H^2(\mathbb T^2,P)$ is isomorphic to $\mathbb Z\oplus\mathbb Z_n$. The group $H^2_P(\mathbb T^2)$ is the subgroup of $H^2(\mathbb T^2,\mathbb R)\simeq\mathbb R$ spanned by $x$ and $y$.

4) For any complete lattice $P$ from series $\mathbb K^2_{m,y;0,x}$, $m\in\mathbb Z$, $m\ge0$, the space $H^2(\mathbb K^2,P)$ is isomorphic to $\mathbb Z_2\oplus\mathbb Z_m$ (where, by definition, $\mathbb Z_0$ is $\mathbb Z$). The group
$H^2_P(\mathbb K^2)$ is trivial.

5) For any complete lattice $P$ from series $\mathbb K^2_{2m,y;1,x}$, $m\in\mathbb Z$, $m\ge0$, the space $H^2(\mathbb K^2,P)$ is isomorphic to $\mathbb Z_{4m}$ (where $\mathbb Z_0=\mathbb Z$). The group $H^2_P(\mathbb K^2)$ is trivial.
\end{theorem}

The result is summarized in table \ref{tab2}.

\begin{table}[h]
\renewcommand{\arraystretch}{1.2} 
\caption{Invariants of compact Lagrangian fibrations}
\label{tab2}
\vskip3mm
\begin{center}
\begin{tabular}{|c|c|c|c|c|c|}
\hline
Base & Series & \multicolumn{2}{|c|}{Nontrivial twistings}
& \multicolumn{2}{|c|}{$H^2(B^2,P)$}
\\
\cline{3-6} & & $H^2(B^2,\mathbb R)$ & $H^2_P(B^2)$ &
$\mathcal Z^2(B^2,P)$ & $\mathcal B^2(B^2,P)$
\\
\hline\hline
$\mathbb T^2$ & $\mathbb T^2_{u,v;w,z}$& $\mathbb R$ &
$u\mathbb Z+v\mathbb Z+w\mathbb Z+z\mathbb Z$ &
$\mathbb Z\oplus\mathbb Z$ & $0$
\\
\hline
$\mathbb T^2$ & $\mathbb T^2_{n,y;x}$ & $\mathbb R$ &
$x\mathbb Z+y\mathbb Z$ & $\mathbb Z\oplus\mathbb Z$ &
$0\oplus n\mathbb Z$
\\
\hline
$\mathbb K^2$ & $\mathbb K^2_{m,y;0,x}$ & $0$ & $0$ &
$\mathbb Z\oplus\mathbb Z$ & $2\mathbb Z\oplus m\mathbb Z$
\\
\hline
$\mathbb K^2$ & $\mathbb K^2_{2m,y;1,x}$ & $0$ & $0$ &
$\mathbb Z\oplus\mathbb Z$ & $4m\mathbb Z\oplus\mathbb Z$
\\
\hline
\end{tabular}
\end{center}
\end{table}

\begin{proof}[of Theorem~\ref{t16}]
1) The base $B^2$ must be one of the following: the plane, the cylinder or the M\"{o}bius strip (see theorem \ref{t15}). This means that the base $B^{2}$ is non-compact. Thus $H^2_P(B^2)=0$. There is no obstructions $H^2(B^2,P)=0$ because the base $B^2$ is homotopic to a one-dimensional CW-complex.

2), 3) See assertion \ref{cl9} (obstructions) and example \ref{ex7} (nontrivial twistings).

4), 5) See assertion \ref{cl10} for primary obstruction classes. There is no nontrivial twistings, since $H^2(\mathbb K^2, \mathbb R)=0$.
\end{proof}

\begin{corollary}
\label{c9}
Two Lagrangian fibrations $\pi_{i}: (M^{4}_{i},
\omega_{i}) \rightarrow B^{2}$ over a non-compact base $B^{2}$ with
complete lattices $P_{i}$ are Lagrangian equivalent if and only if they have
the same lattice $P_{1}=P_{2}$.
\end{corollary}

Summary. All Lagrangian fibrations $\pi\colon(M^4,\omega)\to B^2$ with compact fibres $F$ are classified by the following data: a closed lattice $P$ of rank $2$ on the base $B^2$, a primary obstruction class $\gamma\in H^2(B^2,P)$ and a nontrivial twisting $[\alpha] \in H^2(B^2,\mathbb R)/H^2_P(B^2)$ (see Theorem \ref{t11}).

Any closed lattice $P$ of rank $2$ on a compact two-dimensional surface $B^2$ is complete (see Theorem \ref{t13} and Remark \ref{r6}). Theorem \ref{t15} classifies all complete lattices up to isomorphism. Groups
$$
H^2(B^2,P),
\qquad
H^2(B^2,\mathbb R)/H^2_P(B^2)
$$
are described in theorem \ref{t16}.

This completes the classification of Lagrangian fibrations $\pi\colon(M^4,\omega)\to B^2$ with compact total spaces $M^4$ up to Lagrangian equivalence.

\section{Examples of Lagrangian and almost Lagrangian fibrations}
\label{s5}

In this section we give an example of a nontrivial almost Lagrangian fibration and obtain a full list of Lagrangian fibrations over two-dimensional surfaces up to fiberwise symplectomorphism.

\begin{example}
\label{ex8}
An almost Lagrangian fibration with a nontrivial symplectic obstruction class (see definition \ref{d6}).
In this example we construct a fibration $\pi\colon(M^6,\eta)\to\mathbb T^3$ over the torus $\mathbb T^3$ such that the total space $M^6$ is compact, $d\eta=\pi^*\psi$ for some $3$-form $\psi$, and the symplectic obstruction class  $[\psi]\in H^3(\mathbb T^3,\mathbb R)$ is nontrivial (see assertion \ref{cl21}).

This fibration can be realized as a quotient of a twisted cotangent bundle. Consider a fibration $\pi_0\colon(T^*\mathbb R^3,\,\omega_0+\pi_0^*\varphi)\to\mathbb R^3$. In standard coordinates $(x,y,z,\alpha,\beta,\gamma)$, where $x$, $y$, $z$ are coordinates on the base and $\alpha$, $\beta$, $\gamma$ are coordinates in the fibre, the canonical form is given by $\omega_0=d\alpha\wedge dx+d\beta\wedge
dy+d\gamma\wedge dz$ and the twisting form is given by $\varphi=2x\,dy\wedge\,dz+2y\,dz\wedge\,
dx+2z\,dx\wedge\,dy$. Let Lagrangian isomorphisms $s_1,s_2,s_3$ and $f_1,f_2,f_3$ of the total space $T^*\mathbb
R^3$ be given by the formulas:

\begin{align}
\label{eq19}
s_{1}(x, y, z, \alpha, \beta, \gamma) &= (x, y, z,
\alpha+1, \beta, \gamma),
\\
\label{eq20}
s_{2}(x, y, z, \alpha, \beta, \gamma) &= (x, y, z,
\alpha, \beta+1, \gamma),
\\
\label{eq21}
s_{3}(x, y, z, \alpha, \beta, \gamma) &= (x, y, z,
\alpha, \beta, \gamma+1),
\\
\label{eq22} f_{1}(x, y, z, \alpha, \beta, \gamma) &= (x+1, y, z,
\alpha, \beta+z, \gamma- y),
\\
\label{eq23}
f_{2}(x, y, z, \alpha, \beta, \gamma) &= (x, y+1, z,
\alpha- z, \beta, \gamma+ x),
\\
\label{eq24} f_{3}(x, y, z, \alpha, \beta, \gamma) &= (x, y, z+1,
\alpha+y, \beta-x, \gamma).
\end{align}

Denote by $\pi\colon(M^6,\eta)\to B^3$ the quotient of $\pi_0\colon(T^*\mathbb R^3,\omega_0+\pi^*\varphi)\to\mathbb R^3$ by the action of Lagrangian isomorphisms $s_i$,~$f_j$.

\begin{claim}
\label{cl21}
The fibration $\pi\colon(M^6,\eta)\to B^3$ is an almost Lagrangian fibration with a nontrivial symplectic obstruction.
\end{claim}

\begin{proof}
The proof is in three steps.

\textsl{Step} 1. The fibration $\pi\colon(M^6,\eta)\to B^3$ is a locally trivial fibration over the torus $\mathbb T^3$ with fibre $\mathbb T^3$. The base is the torus, since $f_i$ act as translations on the base $\mathbb R^3$.  Two points in a fibre are equivalent  if and only if they differ by an integer vector $(\alpha,\beta,\gamma)$ (check that each commutator $[f_i,f_j]=f_if_jf_i^{-1}f_j^{-1}$ takes a point $(x,y,z,\alpha,\beta,\gamma)$ to $(x,y,z,\alpha+\widetilde m_1,\beta+\widetilde m_2,\gamma+\widetilde m_3)$ for some integers $\widetilde m_1$, $\widetilde m_2$, $\widetilde m_3$ depending on $i$ and $j$). Hence the fibre is the torus. The fibration is locally trivial, since the quotient of the base is a manifold.

\textsl{Step}~2. The fibration $\pi\colon(M^6,\eta)\to B^3$ is an almost Lagrangian fibration. The form $\omega_0+\pi_0^*\varphi$ can be induced on the quotient space, since $s_i$,~$f_j$ are fiberwise symplectomorphisms.

\textsl{Step}~3. The symplectic obstruction is nontrivial. Let $\psi$ be a $3$-form on the base $B^3$ such that $d\eta=\pi^*\psi$. Then $\displaystyle\int_{B^3}\psi=6$ ($\ne0$), since $d(\omega_0+\pi_0^*\varphi)=\pi^*_0(6\,dx\,\wedge dy\,\wedge dz)$.
\end{proof}
\end{example}

\begin{example}
\label{ex9} A full list of Lagrangian fibrations over two-dimensional surfaces with compact fibres (see Lemmas  \ref{l12}--\ref{l14}). Note that some fibrations in this list are fiberwise symplectomorphic.

\textsl{Fibrations over the torus $\mathbb T^2$}. All these fibrations can be realized as a quotient of a twisted cotangent bundle. Consider a cotangent bundle $\pi_0\colon(\mathbb R^4,\omega_\lambda)\to\mathbb R^2$. In standard coordinates $(x,y,\alpha,\beta)$, where $x$, $y$ are coordinates on the base and $\alpha$, $\beta$ are coordinates in the fibre, we have $\omega_\lambda=d\alpha\wedge dx+d\beta\wedge dy+\lambda\,dx\wedge dy$. Two Lagrangian isomorphisms $s_1$, $s_2$ are defined by the formulas:

\begin{align}
\label{eq25}
s_{1}(x, y, \alpha, \beta) &= (x, y, \alpha+1, \beta),
\\
\label{eq26}
s_{2}(x, y, \alpha, \beta) &= (x, y, \alpha, \beta+1).
\end{align}

\textsl{Fibrations with lattice $\mathbb T^2_{n,y_0;x_0}$}. Let $t,h$ be Lagrangian isomorphisms given by the formulas:

\begin{align}
\label{eq27}
t(x, y, \alpha, \beta) &= \biggl(x+x_{0}, y, \alpha, \beta+
\frac{nm_{0}+n_{0}}{y_{0}}y\biggr),
\\
\label{eq28}
h(x, y, \alpha, \beta) &= \biggl(x+ny, y+y_{0}, \alpha-
\frac{m_{0}}{x_{0}}x, \beta- n \alpha+\frac{nm_{0}}{x_{0}}x\biggr).
\end{align}

By $\pi\colon(M^4_{m_0,n_0},\omega_\lambda)\to\mathbb T^2_{n,y_0;x_0}$ denote the quotient of the fibration $
{\vrule width0pt depth0pt height10pt}  
\pi_0\colon(\mathbb R^4,\omega_\lambda)\to\mathbb R^2$ by Lagrangian isomorphisms $s_1,s_2$ and $t,h$.

\begin{lemma}
\label{l12}
Any Lagrangian fibration $\pi\colon(M^4,\omega)\to\mathbb T^2$ with lattice $P$ from the series $\mathbb T^2_{n,y_0;x_0}$ is fiberwise symplectomorphic to a fibration $\pi\colon(M^4_{m_0,n_0},\omega_\lambda)\to\mathbb T^2_{n,y_0;x_0}$ for some $\lambda\in\mathbb R$ and $m_0,n_0\in\mathbb Z$.
\end{lemma}

\begin{proof}
The proof is in two steps.

\textsl{Step}~1. The fibration $\pi\colon(M^4_{m_0,n_0},\omega_\lambda)\to\mathbb T^2_{n,y_0;x_0}$ is a Lagrangian fibration with lattice $P$ from the series $\mathbb T^2_{n,y_0;x_0}$ and with obstruction $\bigl(\begin{smallmatrix}m_0
\\
n_0+n\mathbb Z\end{smallmatrix}\bigr)\in P_O$. The proof that the fibration is a Lagrangian fibration over the torus with torus fibre is completely analogous to the proof of Assertion \ref{cl21}.

To prove that the lattice of the fibration belongs to the series $\mathbb T^2_{n,y_0;x_0}$, consider the action of isomorphisms $t$ and $h$ on the fibre $F_O$ over the origin $O$.

There exists a section over a $1$-skeleton such that the corresponding obstruction cocycle is equal to $\bigl(\begin{smallmatrix}m_0
\\
n_0\end{smallmatrix}\bigr)$. In this case the $1$-skeleton can be considered as the parallelogram $OO_1O_2O_3$ with vertices $O=(0,0,0,0)$, $O_1=(x_0,0,0,0)$, $O_2=(x_0,y_0,0,0)$, $O_3=(0,y_0,0,0)$ (compare with the parallelogram from section \ref{s3.2}). Consider the section that assigns to sides $OO_1$ and $OO_3$ line segments $OO_1$  and $OO_3$ respectively. Then the section over $O_2O_3$ is the line segment $O_3Q_3$, where $Q_3=(x_0,y_0,-m_0,nm_0)$. Analogously, the section over $O_1O_2$ is the line segment $O_1Q_2$, where $Q_2=(x_0,y_0,0,nm_0+n_0)$. It remains to note that points $Q_2$ and $Q_3$ differ by $(0,0,m_0,n_0)$.

\textsl{Step}~2. Fibrations $\pi\colon(M^4_{m_0,n_0},\omega_\lambda)\to\mathbb T^2_{n,y_0;x_0}$ realize all possible nontrivial shiftings $\pi\colon(M^4_{m_0,n_0},\omega_0)\to\mathbb T^2_{n,y_0;x_0}$. Indeed, $\omega_\lambda-\omega_0=\pi^*(\lambda\, dx\wedge dy)$ and $\displaystyle\int_{\mathbb T^2}{\lambda\,dx\,\wedge dy}=\lambda x_0y_0$.
\end{proof}

\textsl{Fibrations with lattice $\mathbb T^2_{u,v;w,z}$}. Let Lagrangian isomorphisms $t_1$,~$t_2$ be given by the formulas:
\begin{align}
\label{eq29} t_{1}(x, y, \alpha, \beta) &= (x+u, y+v, \alpha- \delta
x, \beta- \delta y),
\\
\label{eq30}
t_{2}(x, y, \alpha, \beta) &= (x+w, y+z,
\alpha+\gamma x, \beta+ \gamma y)
\end{align}
for some $\delta,\gamma\in\mathbb R$ such that the linear combinations $\gamma u-\delta v$ and $\gamma v-\delta z$
are integers. In other words, $\gamma(u,v)+\delta(w,z)=(m_0,n_0)$ for some $m_0,n_0\in\mathbb Z$.

By $\pi\colon(M^4_{m_0,n_0},\omega_\lambda)\to\mathbb T^2_{u,v;w,z}$ denote the quotient of the fibration $\pi_0\colon(\mathbb R^4,\omega_\lambda)\to\mathbb R^2$ by isomorphisms $s_1$,~$s_2$ and $t_1$,~$t_2$.

\begin{lemma}
\label{l13}
Any Lagrangian fibration $\pi\colon(M^4,\omega)\to\mathbb T^2$ with lattice $P$ from the series $\mathbb T^2_{u,v,w,z}$  is fiberwise symplectomorphic to a fibration $\pi\colon(M^4_{m_0,n_0},\omega_\lambda)\to\mathbb T^2_{u,v;w,z}$ for some  $\delta,\gamma\in\mathbb R$.
\end{lemma}

The proof is completely analogous to the proof of Lemma \ref{l12}. The fibration $\pi\colon(M^4_{m_0,n_0},\omega_\lambda)\to\mathbb T^2_{u,v;w,z}$ is a Lagrangian fibration with lattice $P$ from the series $\mathbb T^2_{u,v,w,z}$ and with obstruction $\bigl(\begin{smallmatrix}m_0
\\
n_0\end{smallmatrix}\bigr)\in P_O$.
\smallskip

\textsl{Fibration over the Klein bottle $\mathbb K^2$ with lattice $\mathbb K^2_{m,y_0;\delta,x_0}$}.
All these fibrations over the Klein bottle can be obtained from a fibration $\pi\colon(I^2\times\mathbb T^2,\omega)\to I^2$ over a parallelogram $I^2$ by identifying points over opposite sides by fiberwise isomorphisms $g$ and $h$.

The base $I^2$ is the parallelogram with vertices $(0,0)$, $(x_0,0)$, $\bigl(\frac{m-\delta}2y_0,y_0\bigr)$ and $\bigl(x_0+\frac{m-\delta}2y_0,y_0\bigr)$ for some $x_0,y_0\in\mathbb R$, $m,\delta\in\mathbb Z$. In standard coordinates ($x$, $y$, $\alpha$, $\beta$), where $x$, $y$ are coordinates on the base and $\alpha,\beta\ (\operatorname{mod}1)$ are coordinates in the fibre, the form $\omega$ is $d\alpha\wedge dx+d\beta\wedge dy$. Fiberwise symplectomorphisms $g$, $h$ are given by the formulas:

\begin{align}
\nonumber
h(x, y, \alpha, \beta) &= \biggl(x+my +
\frac{m-\delta}{2}y_{0},\ y + y_{0},
\\
\label{eq31}
&\qquad\qquad \alpha+ \frac{m_{0}}{x_{0}}x +
\biggl(\frac{mm_{0}}{x_{0}} + \frac{n_{0}}{x_{0}}\biggr)y,
\ - m \alpha+ \beta+ \frac{n_{0}}{x_{0}}x\biggr),
\\
\label{eq32} g(x, y, \alpha, \beta)& = \biggl(x + (\delta-m) y +
\frac{m-\delta}{2} y_{0} + x_{0},\ -y+y_{0},\ \alpha,\ (\delta-m)
\alpha-\beta\biggr)
\end{align}
for some $m_0,n_0\in\mathbb Z$.

By $\pi\colon(M^4_{m_0,n_0},\omega)\to\mathbb K^2_{m,y_0;\delta,x_0}$ denote the quotient of the fibration $\pi\colon(I^2\times\mathbb T^2,\omega)\to I^2$ by isomorphisms $g$ and $h$ (more precisely, by identifying points over opposite sides of the parallelogram $I^2$).

The proof of the next lemma is similar to that of Lemma \ref{l12}.

\begin{lemma}
\label{l14}
Any Lagrangian fibration over the Klein bottle is fiberwise symplectomorphic to a fibration $\pi\colon(M^4_{m_0,n_0},\omega)\to\mathbb K^2_{m,y_0;\delta,x_0}$ for some $x_0,y_0\in\mathbb R$, $\delta,m,m_0,n_0\in\mathbb Z$.
\end{lemma}

In this case, $\pi\colon(M^4_{m_0,n_0},\omega)\to\mathbb K^2_{m,y_0;\delta,x_0}$ is a Lagrangian fibration with lattice $P$ from the series $\mathbb K^2_{m,y_0;\delta,x_0}$ and one of the obstruction cocycles of the fibration is equal to $\bigl(\begin{smallmatrix}m_0
\\
n_0\end{smallmatrix}\bigr)\in P_O$. There is no need to take all possible $x_0$, $y_0$, $\delta$, $m$, $m_0$, $n_0$. By theorem \ref{t15}, one may assume that $x_0,y_0>0$, $m>0$, $\delta$ is equal to $0$ or $1$, and $m$ is even for $\delta=1$. And by theorem \ref{t16}, it suffices to consider $m_0\,(\bmod$ $2)$ and $n_0\,(\bmod$ $m)$ when $\delta=0$, and it suffices to consider $m_0\,(\bmod$ $2m)$ and $n_0=0$ when $\delta=1$.

It follows that if $m\ne0$, then it suffices to consider only a finite number of $m_0$,~$n_0$.
\end{example}

\section{Further investigations}
\label{s6}

In this section we suggest some directions for further investigations. In sections \ref{s6.1},~\ref{s6.2} we discuss some possible generalizations of obtained results. We conclude with some problems for the reader (see section~\ref{s6.3}).

There are two natural ways to generalize the classification of Lagrangian fibrations up to Lagrangian equivalence. One may consider other fibrations and one may consider other equivalence relations. Let us consider some examples of possible generalizations.

\subsection{Other fibrations}
\label{s6.1}

This paper classifies fibrations $\pi\colon(M^{2n},\omega)\to B^n$ over a manifold $B^n$ such that the form $\omega$ is non-degenerate, $d\omega=\pi^*\psi$ (for some $3$-form~$\psi$ on the base $B^n$) and the form $\omega$ vanishes on each fibre $F\subset M^{2n}$ (that is $\omega\big|_F\equiv0$). There are some natural ways to generalize Lagrangian fibrations.

\textsl{Fibrations with singularities.} One may allow a fibration to be locally trivial only over an open everywhere dense set $O^n\subset B^n$. The fibration can have singularities over the complement $B^n/O^n$. One may start with non-degenerate singularities because they are classified by the Eliasson theorem (see, for example, \cite{1}). More information about non-degenerate singularities can be found in \cite{11} and \cite{3} (and in other papers by N. T. Zung).

\textsl{Other structures on the total space}.  The classification of Lagrangian fibrations in this paper is based on the Duistermaat construction (see section \ref{s2.1}). One can easily generalize this construction to the following two cases.

1) The form $\omega$ is non-degenerate and $\omega\big|_F\equiv0$. (Recall that we do not need the condition $d\omega=\pi^*\psi$ for the Duistermaat construction. We use it latter to prove that a fibration is Lagrangian equivalent to a twisted cotangent bundle, see Theorem \ref{t6})

2) Instead a $2$-form there is a bivector $\nu$ on the total space of a fibration $\pi\colon(M^{2n},\nu)\to B^n$ (see section \ref{s2.1}).

\textsl{Bases which are not manifolds}.
In order to study non-degenerate singularities we need bases which are not manifolds. By Delzant's theorem the base of a fibrations with elliptic singularities is a manifold with corners (for details see, for example, \cite{12}).

Also, in \cite{4} K. N. Mishachev suggested that one may try to classify fibration over orbifolds.

\subsection{Other classifications}
\label{s6.2}

There are three natural equivalence relations on the set of fibrations with a symplectic structure on the total space.
\begin{itemize}
\item[1.] Lagrangian equivalence $=$ fiberwise symplectomorphism identical on the base.
\item[2.] Lagrangian isomorphism $=$ fiberwise symplectomorphism.
\item[3.] Fiberwise diffeomorphism (not necessary identical on the base).
\end{itemize}

One may try to classify these fibrations up to these equivalences. Also one may classify the total spaces up to
\begin{itemize}
\item[A.] Symplectomorphism (not necessary fiberwise).
\item[B.] Diffeomorphism (not necessary fiberwise).
\end{itemize}

Different classifications of Lagrangian fibrations are also discussed in \cite{13}. Total spaces of Lagrangian fibration with non-degenerate non-hyperbolic singularities are classified up to diffeomorphism in \cite{13}.

\begin{remark}
\label{r12}
Compact Lagrangian fibrations over surfaces were classified up to Lagrangian isomorphism by Mishachev in \cite{4}. Unfortunately, he made several mistakes. The classification of fibrations up to Lagrangian isomorphism is in two steps.

\textsl{Step}~1. Classify closed lattices of rank $2$ up to isomorphism.

\textsl{Step}~2. For each closed lattice $P$ take the quotient of the action of the automorphism group $\mathrm{Aut}_P$ on the set of Lagrangian fibrations with lattice $P$.

In terms of sheafs (see section \ref{s2.4.3} and \S\,\ref{s3}) step 2 is equivalent to the computation of the quotient of the group $H^1(B,\Lambda(T^*\!B/P))$ by the action of $\mathrm{Aut}_P$. Mishachev described the action of the group $\mathrm{Aut}_P$ on the normal subgroup $H^0(B,Z^2B)/d_*(H^0(B,\Gamma(T^*\!B/P)))$ and on the quotient by this subgroup, that is on the kernel of the homomorphism $H^1(B,\Gamma(T^*\!B/P))\stackrel{d_*}\longrightarrow H^1(B,Z^2B)$. But this does not determine the action on $H^1(B,\Lambda(T^*\!B/P))$.

It remains to choose a representative for each element of   $\mathrm{Ker}(H^1(B,\Gamma(T^*\!B/P))\stackrel{d_*}\longrightarrow
H^1(B,Z^2B))$ and then to describe the action of the group $\mathrm{Aut}_P$ on the representatives (from the group $H^1(B,\Lambda(T^*\!B/P))$). In other words (in terms of obstructions), it remain to take a Lagrangian fibration for each obstruction and then to describe the action of $\mathrm{Aut}_P$ on these fibrations.

There is also a (minor) mistake in Mishachev calculations of the group $H^1(B,\Gamma(T^*\!B/P))$, which he denoted by
$H^1(\mathcal T^*\mathcal B/\mathcal P)$. The group $H^1(\mathcal T^*\mathcal B/\mathcal P)$ is not $\mathbb Z_n\oplus\mathbb Z$ but $\mathbb Z\oplus\mathbb Z_n$.
\end{remark}

\subsection{Several problems}
\label{s6.3}

\begin{problem}
\label{pr1}
Classify all Lagrangian fibrations with non-degenerate singularities up to Lagrangian equivalence.
\end{problem}

This is a nontrivial problem. Some facts about local and semi-local structure of non-degenerate singularities can be found in \cite{1} and in papers by N. T. Zung (see, for example, \cite{11}). Fibrations with two types of non-degenerate singularities (non-hyperbolic) are studied in \cite{12} and \cite{13}.

The classification in terms of \v{C}ech cohomology (see section \ref{s2.4.3}) is not an answer
This classification can be roughly outlined as follows. For each locally trivial fibration $\pi\colon(M^{2n},\omega)\to B^n$ with fibre $F$ there exists a cover $\{U_\alpha\}$ of the base $B^n$ such that the induced fibration over each element of the cover is trivial. Denote this trivializations by $\varphi_\alpha\colon\pi^{-1}(U_\alpha)\to U_\alpha\times F$.

Classification in terms of \v{C}ech cohomology states that the fibration is uniquely determined by the set of transition functions $\varphi_\beta\circ\varphi_\gamma^{-1}$. And this transition functions $\varphi_\alpha$ are determined by (Lagrangian) sections of the Lagrangian fibrations, that is by elements of the sheaf $\Lambda(T^*\!B/P)$.

Analogously fibrations with structure group $G$ over base $B$ are classified by \v{C}ech $1$-cohomology
$H^1(B,\mathcal G)$ with coefficients in the sheaf $\mathcal G$ of sections of the trivial fibration $G\times B\to B$.

This group is hard to compute unless it is trivial (for instance, the base is homotopic to a point or a circle,
compare \cite{13}). This yields that using these classifying invariants it is hard to get an explicit list of Lagrangian fibrations.

In conclusion, here are two more problems about Lagrangian fibrations. Recall the plan of classification of Lagrangian fibrations from \S\,~\ref{s3}. The most difficult part of this plan is the classification of lattices.

\begin{problem}
Classify all compact (almost) Lagrangian fibrations over the $n$-torus $\mathbb T^n$.
\label{pr2}
\end{problem}

For $n>2$, there are almost Lagrangian fibrations that are not Lagrangian. In other words there exist almost Lagrangian fibrations with nontrivial symplectic obstructions (see definition \ref{d6}). It is interesting whether all classes $H^3(\mathbb T^n,\mathbb R)$ can be realized by symplectic obstructions.

\begin{problem}
Classify (almost) Lagrangian fibrations over two-dimensional surfaces with fibre $\mathbb R^1\times\mathbb T^1$.
\label{pr3}
\end{problem}

Maybe, the many-valued Morse theory developed by S.P. Novikov (see \cite{14}) would help to solve the last problem.

\smallskip

After the paper had been written, author found out that Lagrangian fibrations over the Klein bottle were independently classified by D. Sepe (see \cite{15}).


\begin{thebibliography}{99}

\bibitem{1} A.\,V.~Bolsinov, A.\,A.~Oshemkov, ``Singularities of integrable Hamiltonian systems'', \textit{Topological methods in the theory of integrable systems}, Cambridge Sci. Publ., Cambridge 2006, pp. 1--67.

\bibitem{2} J.\,J.~Duistermaat, ``On global action-angle coordinates'', \textit{Comm. Pure Appl. Math.}, \textbf{33}:6 (1980), 687--706.

\bibitem{3} N.\,T.~Zung, ``Symplectic topology of integrable Hamiltonian systems.
II:~Topological classification'', \textit{Compositio Math.}, \textbf{138}:2 (2003), 125--156.

\bibitem{4} K.\,N.~Mishachev, ``The classification of Lagrangian bundles over surfaces'', \textit{Differential Geom. Appl.}, \textbf{6}:4 (1996), 301--320.

\bibitem{5} S.-T.~Hu, \textit{Homotopy theory}, Academic Press, New York--London 1959.


\bibitem{6} A.\,T.~Fomenko and D.\,B.~Fuchs (Fuks), \textit{A course in homotopic topology}, Nauka, Moscow 1989 (Russian).

\bibitem{7} B.\,A.~Dubrovin, S.\,P.~Novikov and A.\,T.~Fomenko, \textit{Modern geometry. Methods and applications}, Nauka, Moscow 1986; English transl. as B.\,A.~Dubrovin, A.\,T.~Fomenko, S.\,P.~Novikov, \textit{Modern geometry -- methods and applications. Part~{\rm III}. Introduction to homology theory}, Grad. Texts in Math., vol. 124, Springer-Verlag, New York 1990.

\bibitem{8} J.~Milnor, ``On the existence of a connection with curvature zero'', \textit{Comment. Math. Helv.}, \textbf{32} (1958), 215--223

\bibitem{9} D.~Freid, W.~Goldman and M.\,W.~Hirsh, ``Affine manifolds with nilpotent holonomy'', \textit{Comment. Math. Helv.}, \textbf{56}:4 (1981), 487--523.

\bibitem{10} W.~Magnus, A.~Karrass and D.~Solitar, \textit{Combinatorial group theory: Presentations of groups in terms of generators and relations}, Interscience Publ., New York--London--Sydney 1966.

\bibitem{11} N.\,T.~Zung, ``Symplectic topology of integrable Hamiltonian systems.
I:~Arnold--\allowbreak Liouville with singularities'', \textit{Compositio Math.}, \textbf{101}:2 (1996), 179--215.

\bibitem{12} M.~Symington, ``Four dimensions from two in symplectic topology'', \textit{Topology and geometry of manifolds} (University of Georgia, Athens, GA, USA, 2001), Proc. Sympos. Pure Math., vol. 71, Amer. Math. Soc., Providence, RI 2003, pp. 153--208.

\bibitem{13} N.\,C.~Leung, M.~Symington, ``Almost toric symplectic four-manifolds'', \textit{J. Symplectic Geom.}, \textbf{8}:2 (2010), 143--187.

\bibitem{14} S.\,P.~Novikov, ``The Hamiltonian formalism and a many-valued analogue of Morse theory'', \textit{Upekhi Mat. Nauk}, \textbf{37}:5 (1982), 3--49; English transl. in \textit{Russian Math. Surveys}, \textbf{37}:5 (1982), 1--56.

\bibitem{15} D.~Sepe, \textit{Classification of Lagrangian fibrations over a Klein bottle}, 
\texttt{arXiv:0909.2982 [math.SG]}

\end{thebibliography}
\end{document}